\documentclass{amsart}

\usepackage{amsmath,amssymb,amsthm}
\usepackage{graphicx,tikz}

\newtheorem*{thm}{Theorem}
\newtheorem*{proposition}{Proposition}

\newtheorem*{lem}{Lemma}

\theoremstyle{definition}

\theoremstyle{remark}

\begin{document}

\title[]{A Nonlocal Functional Promoting Low-Discrepancy Point Sets}
\keywords{Low discrepancy sequence, gradient flow, energy functional.}
\subjclass[2010]{11L03, 42B05, 82C22.}

\author[]{Stefan Steinerberger}
\address{Department of Mathematics, Yale University, New Haven, CT 06511, USA}
\email{stefan.steinerberger@yale.edu}
\thanks{The author is supported by the NSF (DMS-1763179) and the Alfred P. Sloan Foundation.}

\begin{abstract} Let $X = \left\{x_1, \dots, x_N\right\} \subset \mathbb{T}^d \cong [0,1]^d$ be a set of $N$ points in the $d-$dimensional torus that we want to arrange as regularly possible.  The purpose of this paper is to introduce the energy functional 
$$ E(X) = \sum_{1 \leq m,n \leq N \atop m \neq n} \prod_{k=1}^{d}{ (1 - \log{\left(2 \sin{ \left( \pi |x_{m,k} - x_{n,k} |\right)} \right)})}$$
and to suggest that moving a set $X$ into the direction $-\nabla E(X)$ may have the effect of increasing regularity of the set in the sense of decreasing discrepancy. We numerically demonstrate the effect for Halton, Hammersley, Kronecker, Niederreiter and Sobol sets. Lattices in $d=2$ are critical points of the energy functional, some (possibly all) are strict local minima.
\end{abstract}

\maketitle

\vspace{-10pt}

\section{Introduction}
\subsection{Introduction.} This paper is partially motivated by earlier results about how to distribute points on a manifold in a regular way. One idea (from \cite{sachs, stein}) is to not construct these points a priori but instead use (local) minimizers of an energy functional. For example, suppose we want to distribute $N$ points on the two-dimensional torus $\mathbb{T}^2$ in a way that is good for numerical integration. One way of doing this is by trying to find local minimizers of the energy functional
$$ F(X) = \sum_{1 \leq m,n \leq N \atop m \neq n}{ e^{- cN^{-1} \|x_i - x_j\|^2}},$$
where $c \sim 1$ is a constant.
These point configurations are empirically shown \cite{sachs} to be better at integrating trigonometric polynomials than commonly used classical constructions, the reason for that being a connection between the Gaussian and the heat kernel (which, in itself, can be interpreted as a mollifier in Fourier space dampening high oscillation). This method is also geometry independent and works on general compact manifolds (with $\|x_i -x_j\|$ replaced by the geodesic distance).

\subsection{The problem.} We were curious whether there is any way to proceed similarly in the problem of finding low-discrepancy sets of points. Suppose $X \subset [0,1]^2$ is a set $\left\{x_1, \dots, x_N\right\}$ of $N$ distinct points. A classical question is how would to arrange them so as to minimize the star discrepancy $D_N^*(X)$ defined by
$$ D^*_N(X)  = \max_{0 \leq x,y \leq 1} \left| \frac{ \# \left\{1 \leq i \leq N: x_{i,1} \leq x \wedge x_{i,2} \leq y\right\}}{N} - xy \right|.$$
 A seminal result of Schmidt \cite{sch} is
$$D_N^* \gtrsim \frac{\log{N}}{N}.$$
Many constructions of sets are known that attain this growth, we refer to the classical textbooks by Dick \& Pillichshammer \cite{dick}, Drmota \& Tichy \cite{drmota} and Kuipers \& Niederreiter \cite{kuipers} for descriptions. Some of the classical configurations are also used as examples in this paper. The  problem is famously unsolved in higher dimensions where the best known constructions \cite{halton, hammersley, nied, niederreiter, sobol} satisfy $D_N \lesssim (\log{N})^{d-1} N^{-1}$ but no matching lower bound exists (see  \cite{bil1, bil3, bil4}). Indeed, there is not even consensus as to whether the best known constructions attain the optimal growth or whether there might be more effective constructions in $d \geq 3$.

\begin{figure}[h!]
\begin{minipage}[l]{.49\textwidth}
\includegraphics[width = 5cm]{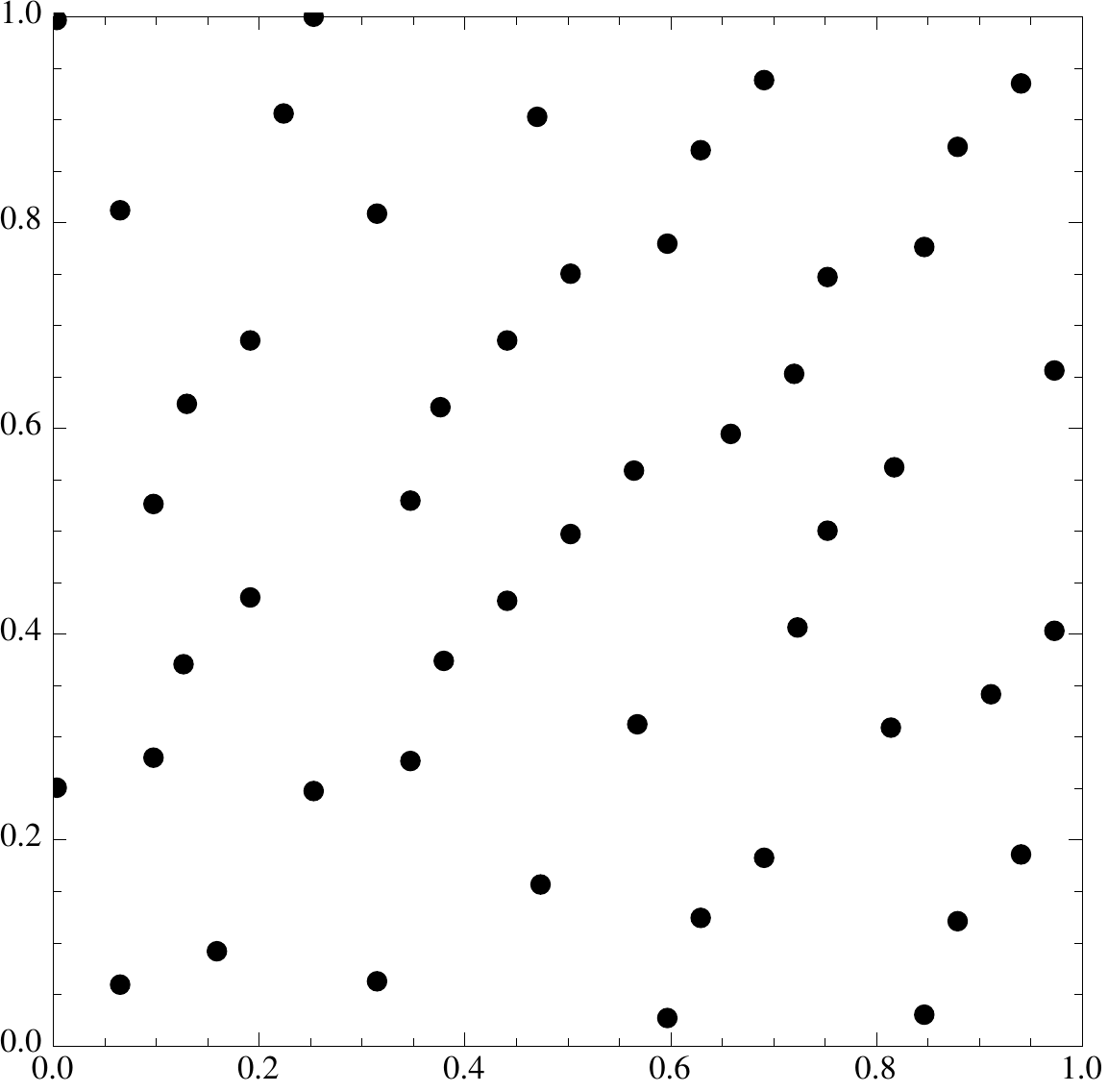} 
\end{minipage} 
\begin{minipage}[r]{.49\textwidth}
\includegraphics[width = 5cm]{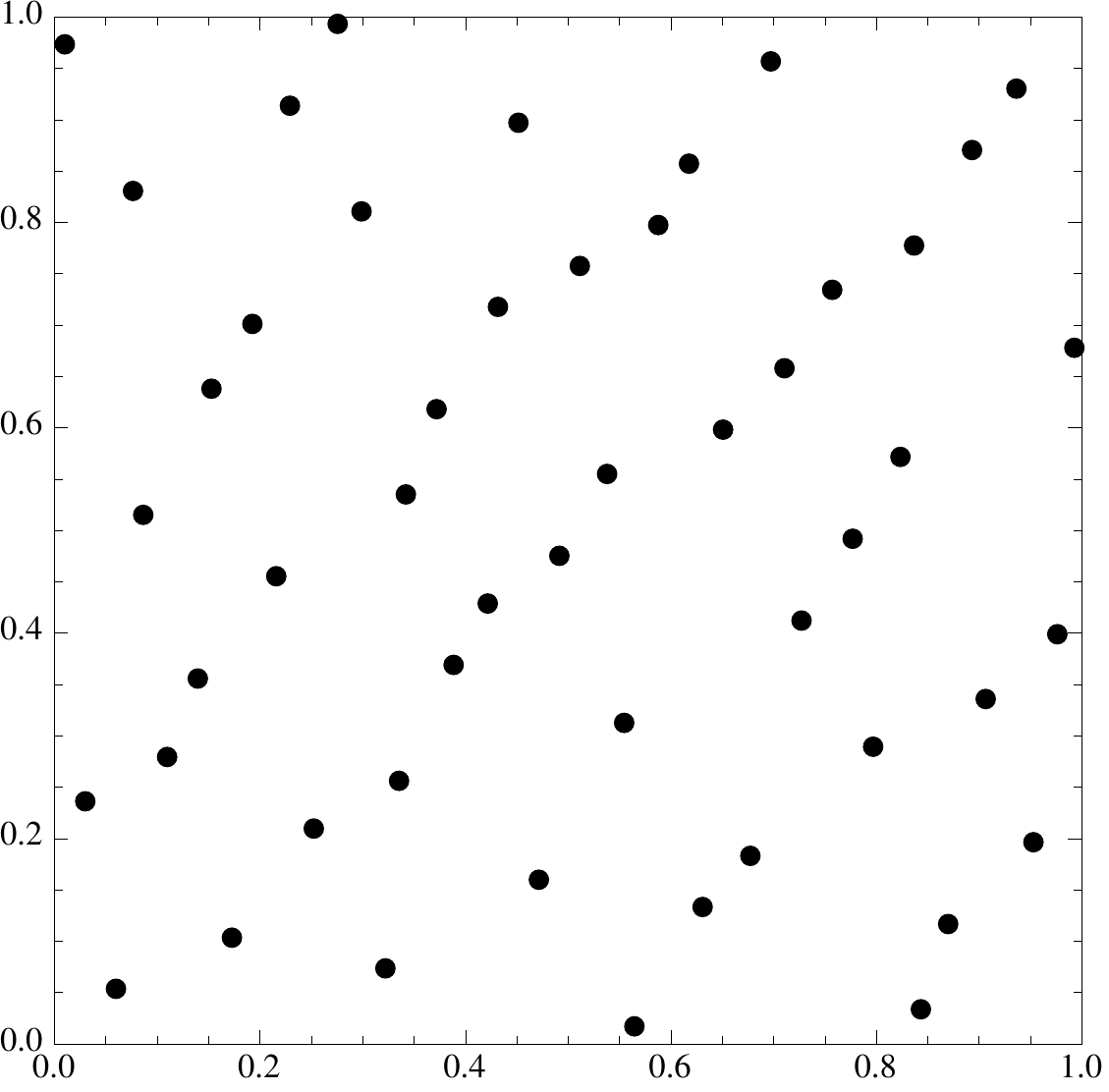} 
\end{minipage} 
\caption{Left: 50 points of a Niederreiter sequence with $D_N^* \sim 0.082$. Right: gradient flow produces a (similar) set, $D_N^* \sim 0.061$.}
\end{figure}

We were interested in whether it is possible to assign a notion of 'energy' to a set of points that vaguely corresponds to discrepancy in the sense that moving the points in such a way that perturbations of the points decreasing the energy also often decrease discrepancy. What would be of interest is a notion of energy that is
\begin{enumerate}
\item fast to compute
\item often helpful in improving existing point sets
\item and may have the potential to lead to new constructions.
\end{enumerate}
We believe this questions to be of some interest. The purpose of this paper is to derive one functional that seems to work very well in practice. Indeed, it works strikingly well: when applied to the classical low discrepancy constructions, it always seems to further decrease discrepancy (though sometimes, when the sets are already well distributed, only by very little). We provide a heuristic explanation in \S 3.3.
There might be many other such functionals (possibly related to different kinds of mathematics, e.g. \cite{ba, mo, os}) and we believe that constructing and understanding them could be quite interesting indeed.	
\begin{quote}
\textbf{Open Problem.} Construct other energy functionals whose gradient flow has a beneficial effect on discrepancy. What can be rigorously proven? Can they be used for numerical integration? How do they scale in the dimension?
\end{quote}

\subsection{Related results.}
We emphasize that this open problem stated in \S 1.2. is \textit{wide} open. In particular, we do not claim that our energy functional is necessarily the most effective one. Our functional certainly seems natural in light of our derivation; moreover, the author recently used it \cite{steind} to define sequences $(x_n)_{n=1}^{\infty}$ whose discrepancy seems to be extremely good when compared to classical sequences (however, the only known bound for these sequences is currently $D_N \lesssim N^{-1/2} \log{N}$). Nonetheless, there may be other functionals that are as natural and even more efficient. As an example of another functional that could be of interest, we mention Warnock's formula 
\cite[Lemma 2.14]{mat} for the $L^2-$discrepancy
$$ L^2(X)^2 = \frac{1}{3^s} - \frac{2}{N}\sum_{n=1}^{N}{ \prod_{i=1}^{d}{ \frac{1-x_{n,i}^2}{2}}} + \frac{1}{N^2} \sum_{n,m=1}^{N} \prod_{i=1}^{d} \min(1-x_{n,i}, 1-x_{m,i})$$
This could be used to define a gradient flow (where one has to be a bit careful with the non-differentiability of the minimum). A similar construction is presumably possible at a much greater level by using integration formulas in reproducing kernel Hilbert spaces (see e.g. \cite{fritz}). We recall that, if we sample in $(x_j)_{j=1}^{n}$ with weights $(w_j)_{j=1}^{n}$, then the worst case error in a reproducing kernel Hilbert space is given by the formula
\begin{align*}
 \mbox{worst-case error} &= \sum_{i,j=1}^{n}{w_i w_j K(x_i, x_j)} - 2 \sum_{j=1}^{n}{w_j} \int_{\Omega} K(x_j, y) d\mu(y) \\
&+ \int_{\Omega} \int_{\Omega} K(x,y) d\mu(x) d\mu(y).
\end{align*}
Functionals of this flavor might be amenable to a gradient flow approach at a great level of generality, however, this is outside the scope of this paper.

\begin{figure}[h!]
\begin{minipage}[l]{.49\textwidth}
\includegraphics[width = 5cm]{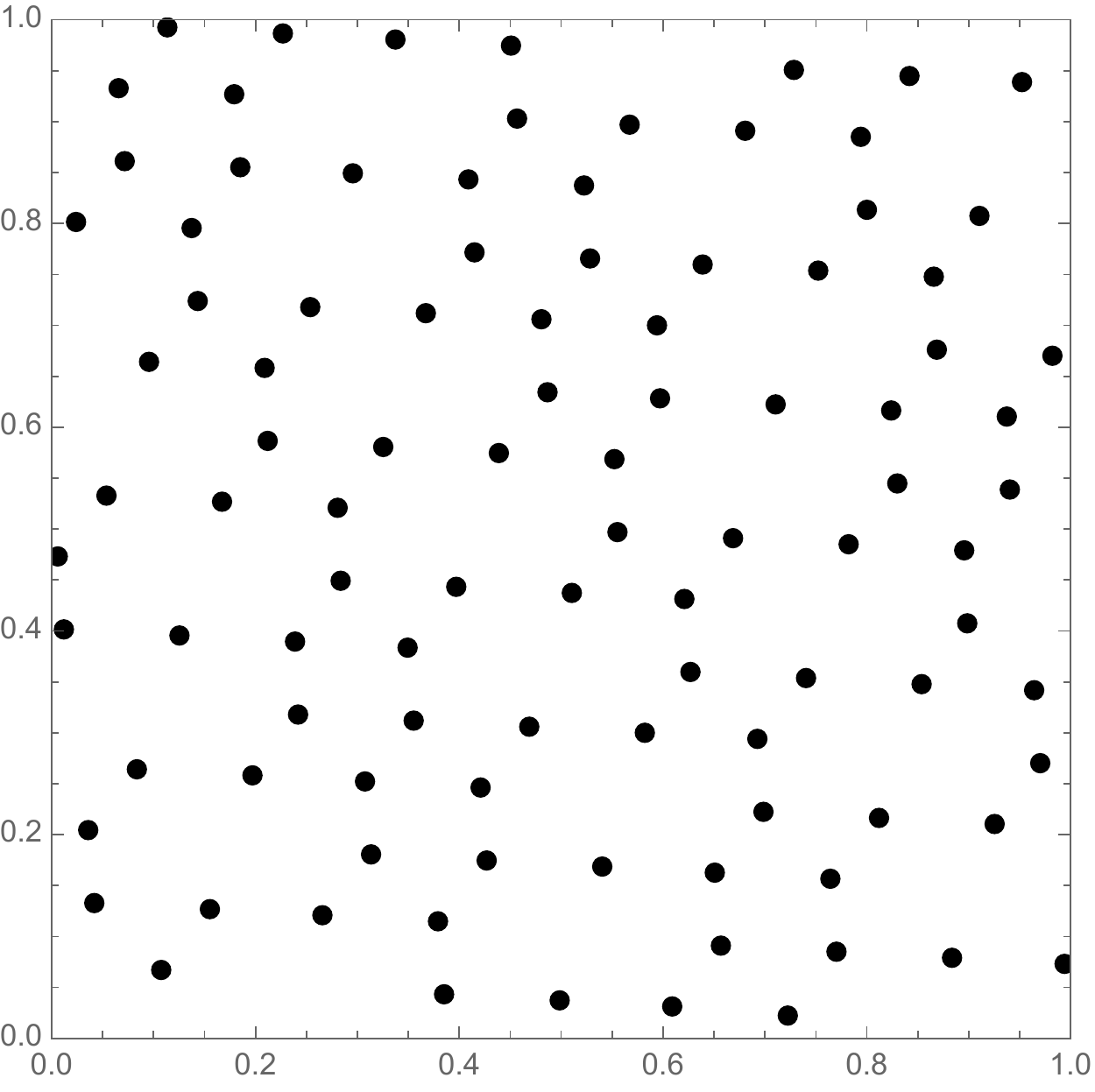} 
\end{minipage} 
\begin{minipage}[r]{.49\textwidth}
\includegraphics[width = 5cm]{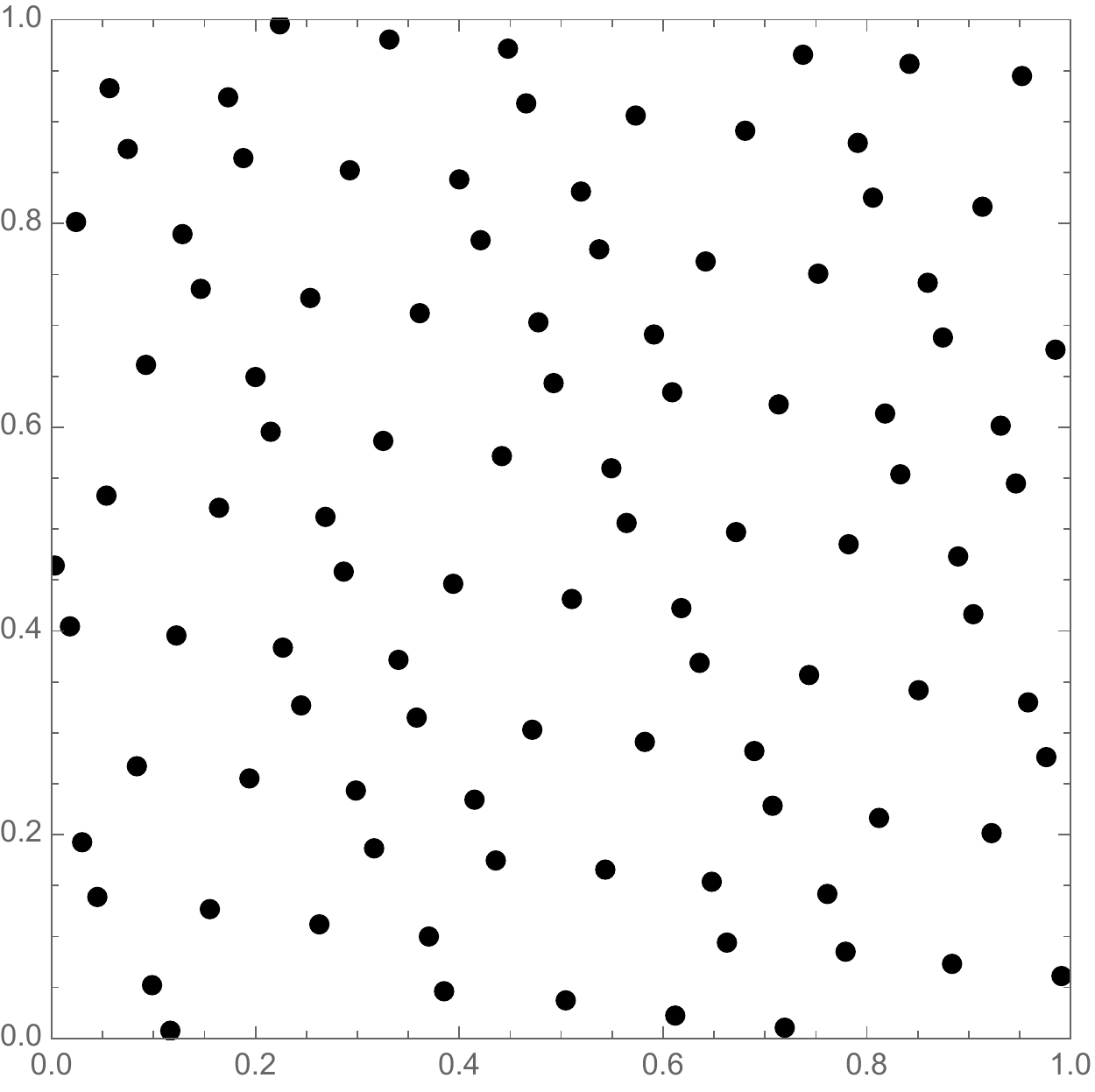} 
\end{minipage} 
\caption{Left: the set $\left\{ (\left\{\sqrt{2}n\right\}, \left\{\sqrt{\pi} n \right\} ): 1 \leq n \leq 100\right\}$ with  $D_N^* \sim 0.04$. Right: evolution of the gradient flow leads to a set with discrepancy $D_N^* \sim 0.03$.}
\end{figure}

To the best of our knowledge, the approach outlined in this paper as well as the associated functional is new. There is a broad literature around the underlying problem of construction of low-discrepancy sequences by various means. Traditional results where mainly concerned with asymptotic results (see e.g. \cite{dick, drmota, kuipers}). These constructions often have implicit constants that grow very quickly in the dimension; the search for results that are effective for a small number of points initiated a fertile area of research \cite{aisti, doerr00, heinrich, hinrich}. Even the mere task of computing discrepancy in high dimensions is nontrivial \cite{doerr2, gnewuch, gnewuch2}. We are not aware of any optimization algorithms that take an explicit set of points and then induce a gradient flow to try to decrease the discrepancy.

\subsection{Outline of the paper.} \S 2 introduces the energy functional and the main result. \S 3 explains how the energy functional was derived, describes the one-dimensional setting and relates it to Fourier analysis. A proof of the main result is given in \S 4. Numerical examples of how the energy functional acts on well-known constructions are given throughout the paper -- these examples are all two-dimensional (for simplicity of exposition). 

\begin{table}[h!]
\begin{tabular}{ l | c|  c |  c }
 Type of Sequences & $N$  & Discrepancy $D_N(X_N)$ & $D_N$ after Optimization \\[0.05cm]
Niederreiter sequence & 50 & 0.082 & 0.061    \\[0.05cm]
Hammersley (base 3) & 50   & 0.064  &  0.042     \\[0.05cm]
Sobol & 50  & 0.063 &  0.057   \\[0.05cm]
Halton (base 2 and 5) & 64   &  0.064 & 0.045    \\[0.05cm]
random points & 100   & 0.12  &  0.05     \\[0.05cm]
Halton (base 2 and 3) & 128 & 0.032 & 0.025   \\[0.05cm]
Niederreiter in $[0,1]^3$ & 50 & 0.098 & 0.093 \\[0.05cm]
$\mbox{vdc}_2 \times \mbox{vdc}_3 \times \left\{ \pi n\right\}$ & 100 & 0.074 & 0.066
\end{tabular}
\vspace{3pt}
\caption{Examples shown in this paper.}
\end{table}

We emphasize that the examples of point sets are all essentially picked at random, the functional does seem to work at an \textit{overwhelming} level of generality and we invite the reader to try it on their own favorite sets.

\section{An energy functional}
\subsection{The functional.} Given a set $X = \left\{x_1, \dots, x_N\right\} \subset \mathbb{T}^d \cong [0,1]^d$ of $N$ points in the $d-$dimensional torus where each point is given by 
$$ x_n = (x_{n,1}, \dots, x_{n,d}) \in \mathbb{T}^d,$$
we introduce the energy function $E:([0,1]^d)^N \rightarrow \mathbb{R}$ via
$$ E(X) = \sum_{1 \leq m,n \leq N \atop m \neq n} \prod_{k=1}^{d}{ (1 - \log{\left(2 \sin{ \left( \pi |x_{m,k} - x_{n,k} |\right)} \right)})}.$$
We note that, for $0 \leq x,y \leq 1$ we have that
$$1 - \log{(2 \sin{ \pi |x - y|})} \geq 1 - \log{2}$$
and so every term in the product is always positive. We also note that if two different points $x_i, x_j$ have the same $k-$th coordinate, then the functional is not defined and we set $E(X) = \infty$ in that case. In practice, we can always perturb points ever so slightly to avoid that scenario. We note that the functional has an interesting structure: it very much likes to avoid having too many points that have very similar coordinates. This makes sense since such points can be easily captured by a thin (hyper-)rectangle. We now first discuss how to actually minimize it in practice and then discuss our main result.

\begin{figure}[h!]
\begin{minipage}[l]{.49\textwidth}
\includegraphics[width = 5.5cm]{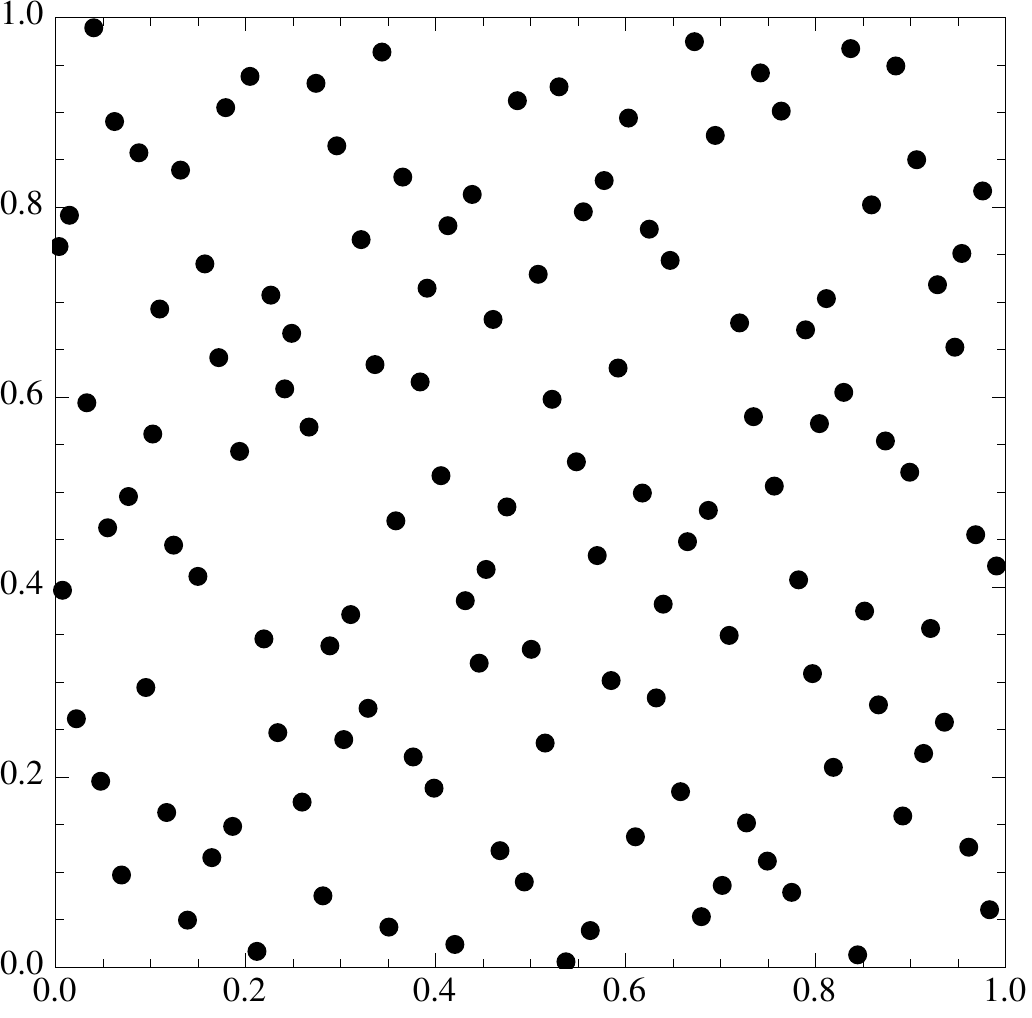} 
\end{minipage} 
\begin{minipage}[r]{.49\textwidth}
\includegraphics[width = 5.5cm]{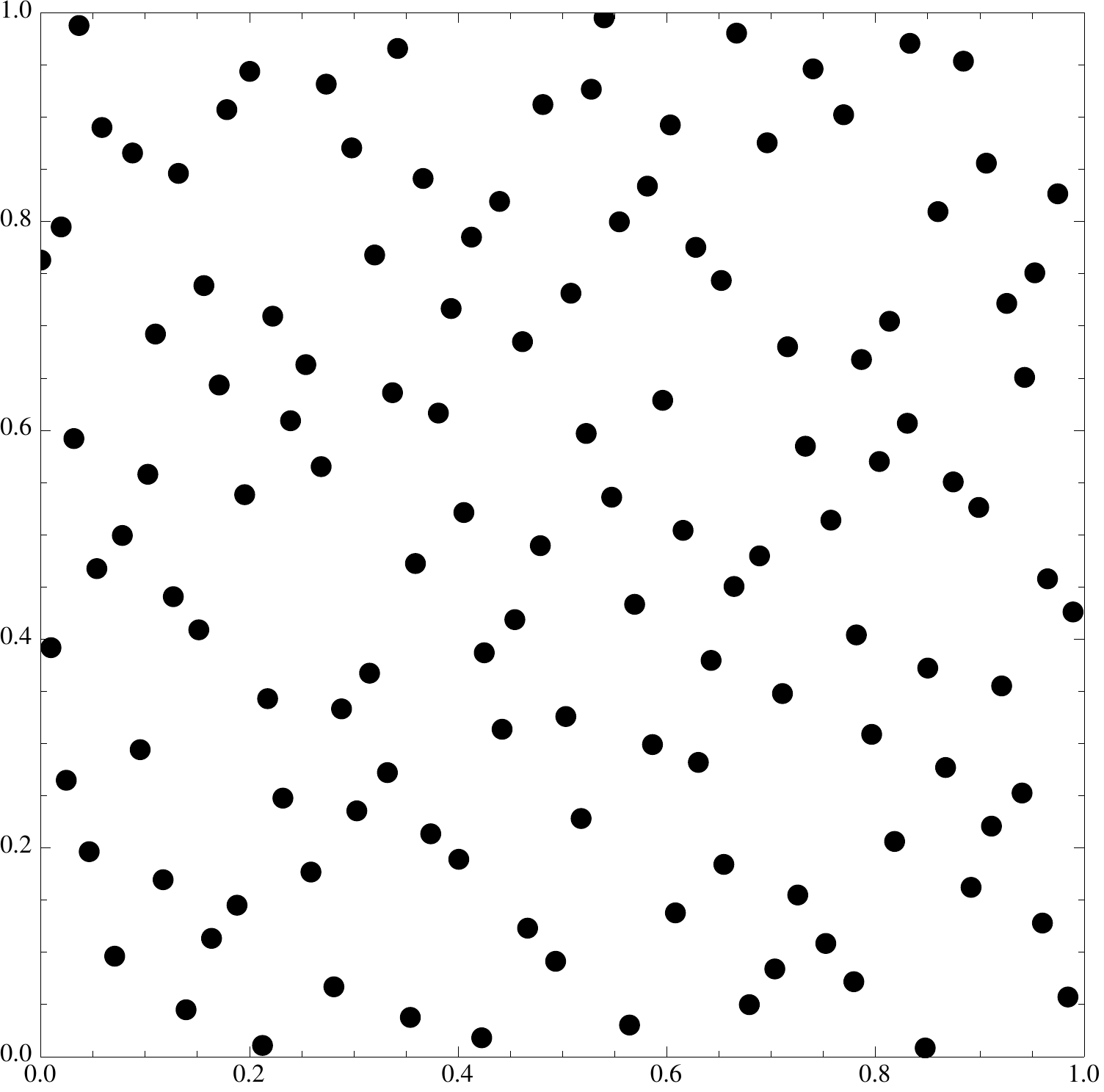} 
\end{minipage} 
\caption{Left: 128 points of the Halton sequence in base 2 and 3 having $D_N^* \sim 0.032$. Right: evolution of the gradient flow changes the set a tiny bit to one with discrepancy $D_N^* \sim 0.025$.}
\end{figure}

\subsection{How to compute things.}
We are using the standard gradient descent: if $f:\mathbb{R}^d \rightarrow \mathbb{R}$ is a differentiable function, gradient descent is trying to find a (local) minimum by defining an iterative sequence of points via
$$ x_{n+1} = x_{n} - \alpha \nabla f(x_n),$$
were $\alpha > 0$ is the step-size. This is exactly how we proceed as well.
The gradient $\nabla E$ can be computed explicitly and
$$ \frac{\partial E}{\partial x_{n, i}} =  \sum_{m=1 \atop m \neq n}^{N} \left( \prod_{k=1 \atop k \neq i}^{d}{ (1 - \log{\left(2 \sin{ \left( \pi |x_{m,k} - x_{n,k} |\right)} \right)})} \right) h(x_{n,i} - x_{m,i}),$$
where
$$ h(x) =  - \pi \cot{(x)} \mbox{sign}(x).$$
This allows us to compute
$$ \frac{\partial E}{\partial x_{n}} = \left( \frac{\partial E}{\partial x_{n, 1}}, \frac{\partial E}{\partial x_{n, 2}}, \dots, \frac{\partial E}{\partial x_{n, d}} \right)$$
which is the infinitesimal direction in which we have to move $x_n$ to get the largest increase in the energy functional. Since we are interested in decreasing it, we replace
$$ x_n \leftarrow x_n - \alpha  \frac{\partial E}{\partial x_{n}}.$$

The algorithm is somewhat sensitive to the choice of $\alpha$ (this is not surprising and a recurring theme for gradient methods): it has to be chosen so small that the first order approximation
is still somewhat valid, however, if it is chosen too small, then convergence becomes very slow and one needs more iterations to converge. In practice, for point sets containing $\sim 100$ points, we worked
with $\alpha \sim 10^{-5}$ which usually leads to a local minimum within less than a hundred iterations. The cost of computing a gradient step is of order $\mathcal{O}(N^2 d)$ when $N \geq d$ and thus not at all unreasonable. There are presumably ways of optimizing both the choice of $\alpha$ as well as the cost of computing the energy (say, by fast multipole techniques) but this is beyond the scope of this paper.

\begin{figure}[h!]
\begin{minipage}[l]{.49\textwidth}
\includegraphics[width = 5.5cm]{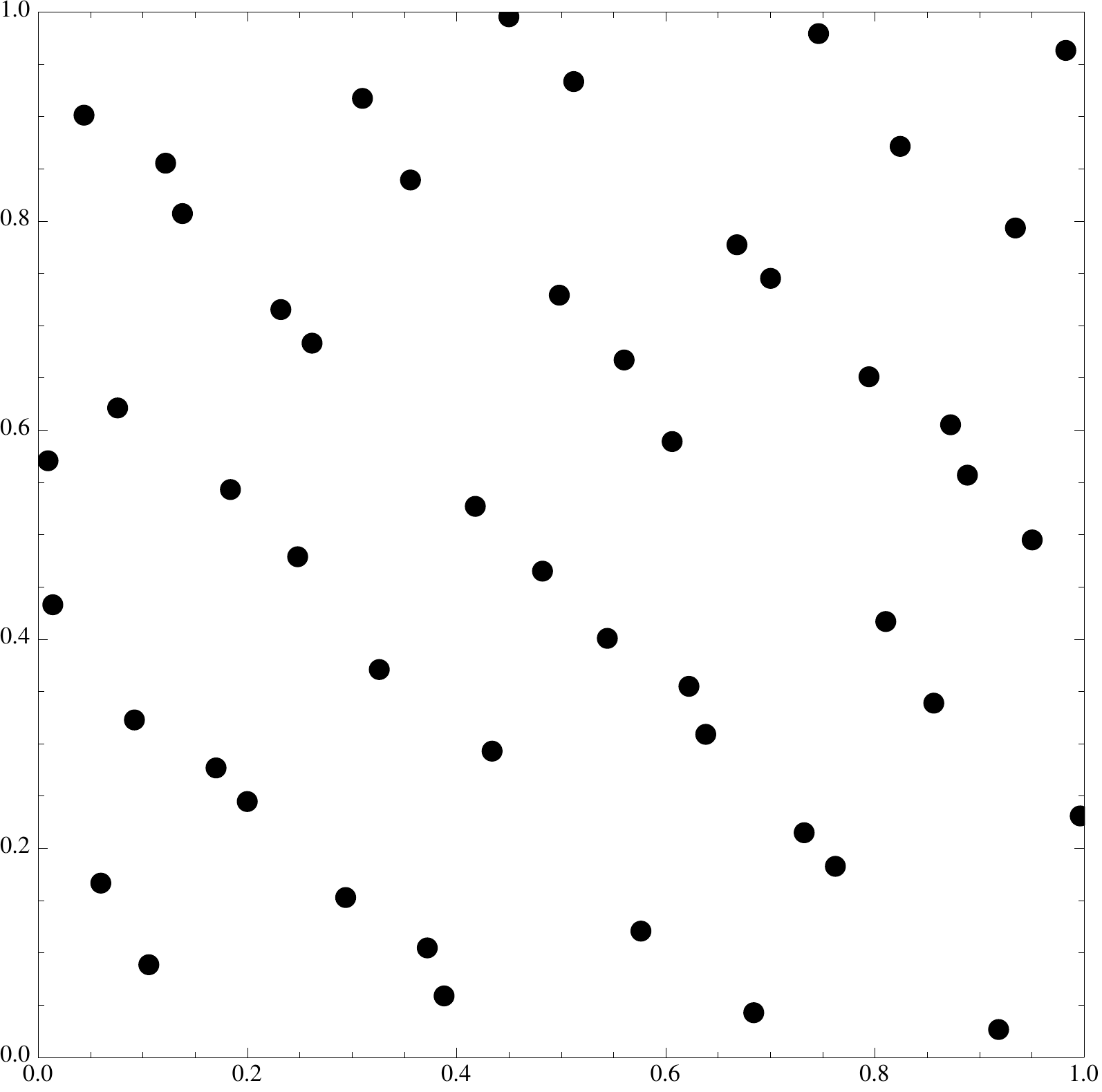} 
\end{minipage} 
\begin{minipage}[r]{.49\textwidth}
\includegraphics[width = 5.5cm]{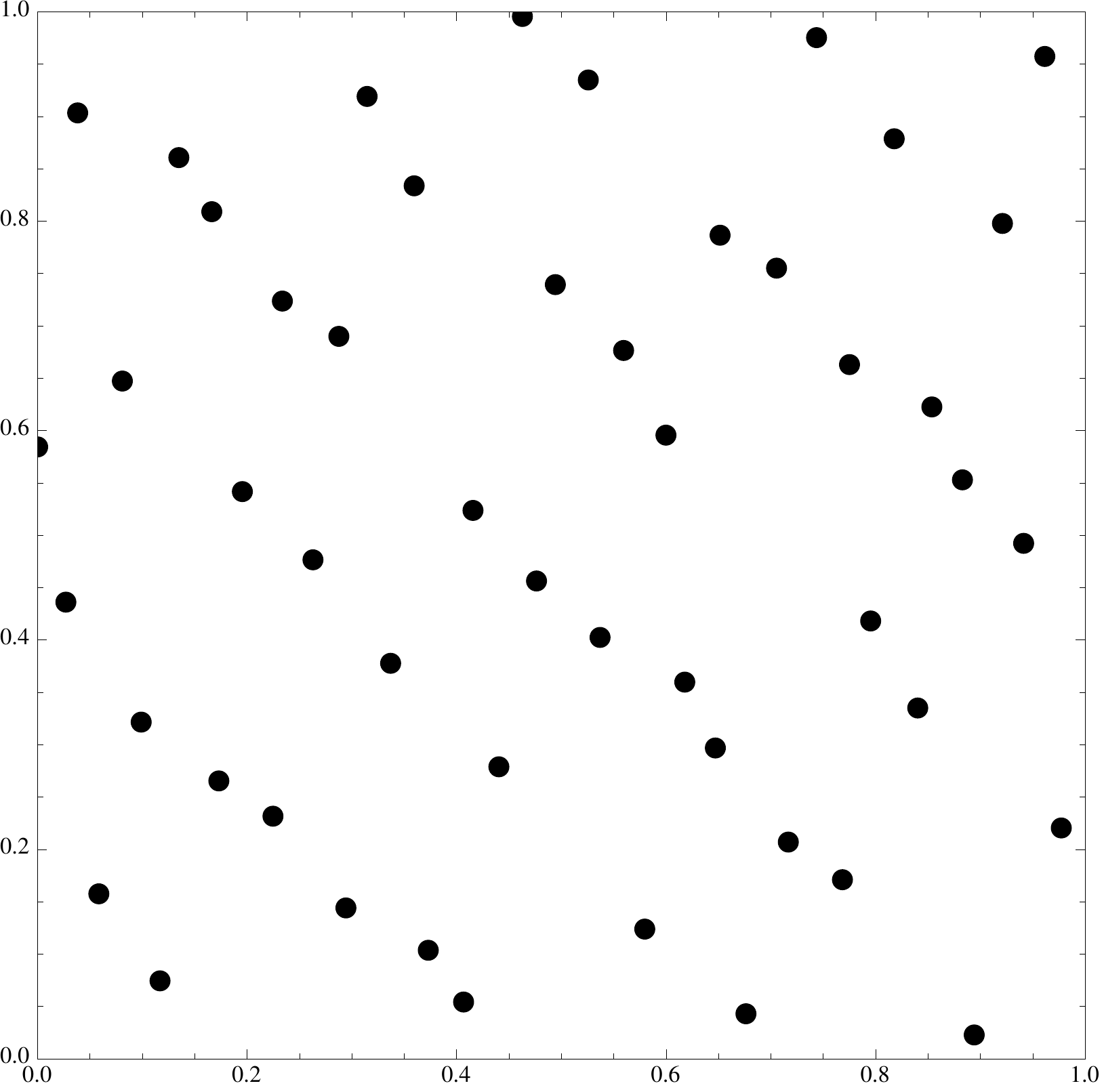} 
\end{minipage} 
\caption{Left: 50 points of a Sobol sequence with $D_N^* \sim 0.063$. Right: evolution of the flow leads to a set with $D_N^* \sim 0.057$.}
\end{figure}

\subsection{Lattices.}
We observe that if the initial point set is already very well distributed, then minimizing the energy tends to have very little effect on both the set and the discrepancy. There is one 
setting where this behavior is especially pronounced. We will consider lattice rules of the type
$$ X_N = \left\{ \left( \frac{n}{N}, \left\{ \frac{a n}{N}\right\} \right): 0 \leq n \leq N-1\right\},$$
where $a,N \in \mathbb{Z}$ are coprime and $\left\{ x \right\} = x - \left\lfloor x \right\rfloor$ is the fractional part. Lattice rules are classical examples of sequences with small discrepancy, we refer to \cite{dick, drmota, kuipers} and refer to \cite{fritz2, fritz3} for examples of more recent results.
\begin{thm}  Every lattice rule $X_N$ is a critical point of the energy functional. Moreover, if $a^2 \equiv 1~(\mbox{mod}~N)$, then $X_N$ is a strict local minimum.
\end{thm}
We understand critical point in the following sense: if we fix all but one point and then move the one point distance $\varepsilon$, then the energy changes by a factor proportional to $\varepsilon^2$. If $a^2 \equiv 1~(\mbox{mod}~N)$, then the energy changes like $\sim c \varepsilon^2$ for some $c>0$. Some restriction like this is clearly necessary since, if we move all the points by the same fixed vector, the energy remains unchanged. Nonetheless, we expect stronger statements to be true.
We also do not know whether the condition $a^2 \equiv 1~(\mbox{mod}~N)$ is necessary, it seems like it should not be; we comment on this at the end of the paper. Several of the classical point
sets (i.e. Sobol sequences) barely move under the gradient flow -- is it maybe true that many classical sequences have a local minimum nearby?

\begin{figure}[h!]
\begin{minipage}[l]{.49\textwidth}
\includegraphics[width = 5cm]{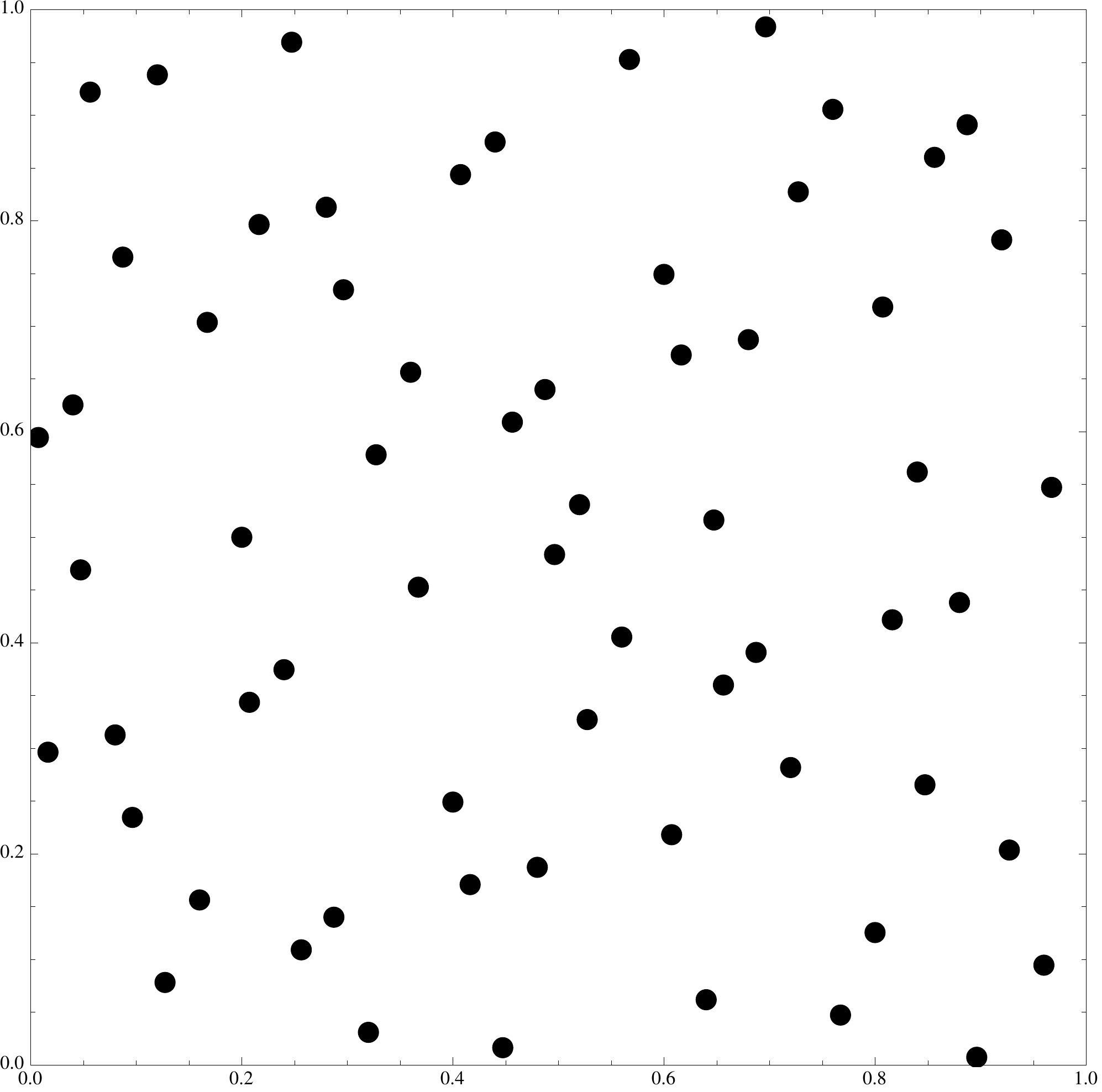} 
\end{minipage} 
\begin{minipage}[r]{.49\textwidth}
\includegraphics[width = 5cm]{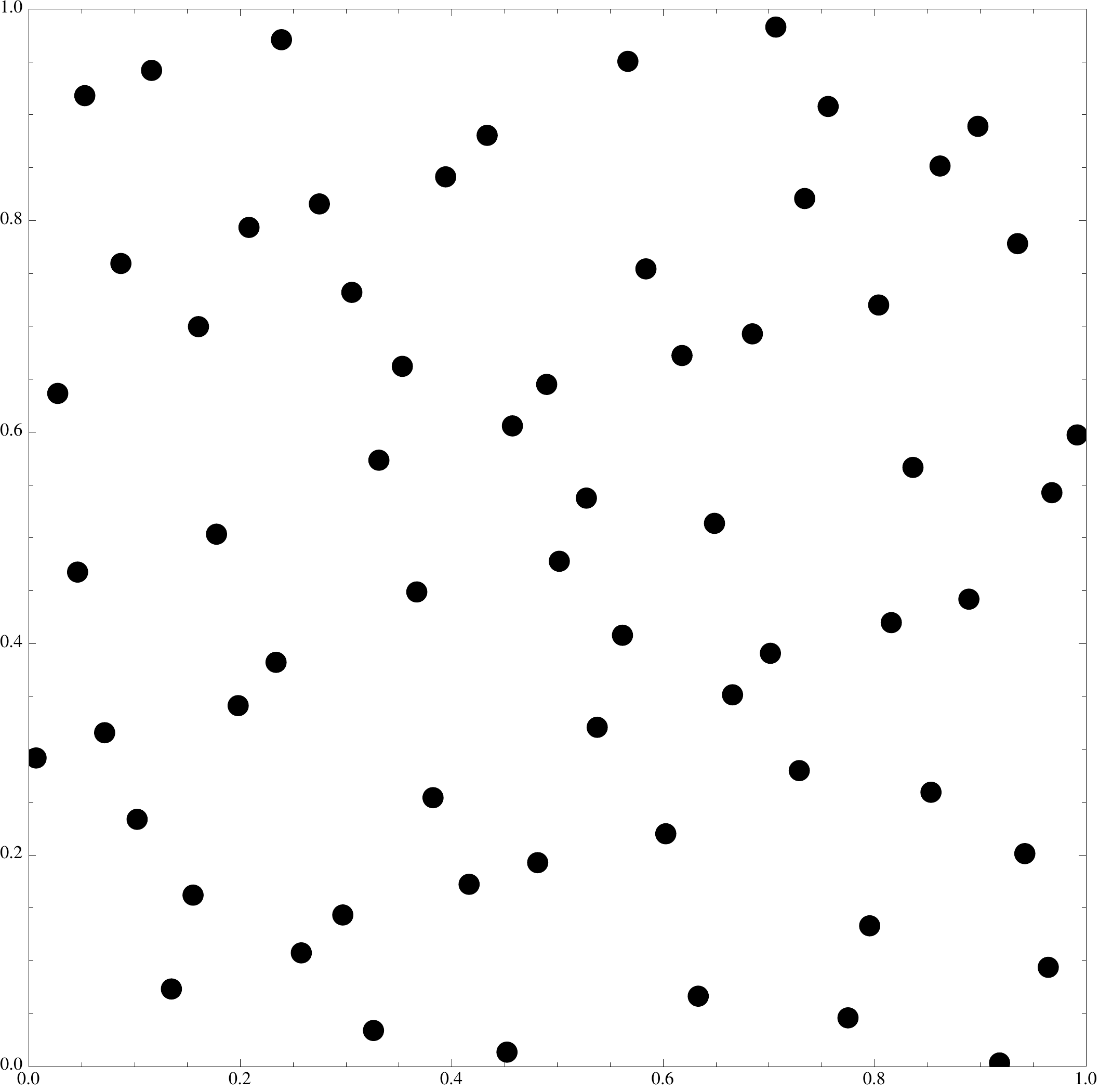} 
\end{minipage} 
\caption{Left: 64 Halton points (base 2, 5) with $D_N^* \sim 0.065$. Right: the gradient flow leads to a set with $D_N^* \sim 0.045$.}
\end{figure}

\subsection{Related functionals.} One question of obvious interest is whether there are related functionals. We point out that our functional is part of a natural 1-parameter family of functionals that are naturally defined via certain fractional integral operators. This is \textit{not} how our functional was originally derived (that derivation can be found in \S 3 and is based on the Erd\H{o}s-Turan inequality) but may provide an interesting avenue for further research. We note that our approach to the Erd\H{o}s-Turan inequality involves an application of a Cauchy-Schwarz inequality that could also be done in a different fashion and this would, somewhat naturally, lead to the inverse fractional Laplacian. We quickly introduces this fascinating object here and then mention explicitly in the proof how one could deviate from the derivation. A full exploration of this case is outside the scope of this paper.
If $f:\mathbb{T}^d \rightarrow \mathbb{R}$ is sufficiently smooth, then we can differentiate term by term and obtain, for any $s \in \mathbb{N}$,
$$ (-\Delta)^s f = \sum_{k \in \mathbb{Z}^d}{ \widehat{f}(k) e^{2\pi i \left\langle k, x \right\rangle}} =  \sum_{k \in \mathbb{Z}^d \atop k \neq 0}{(2\pi \|k\|)^{2s} \widehat{f}(k) e^{2\pi i \left\langle k, x \right\rangle}}.$$
However, as is easily seen, this definition actually makes sense for $s \in \mathbb{R}$: if $s$ is positive, then we require that $\widehat{f}(k)$ decays sufficiently quickly for the sum to be defined. If $s$ is negative, then it suffices to assume that $f\in L^2(\mathbb{T}^d)$ since $\| (-\Delta)^s f \|_{L^2} \leq \|f\|_{L^2}$ for all $s<0$ (we refer to \cite{luz} for an introduction into the fractional Laplacian on the Torus). We will now compute $(-\Delta)^{-1/2} \delta_0$, where $\delta_0$ is a Dirac measure in 0 in $\mathbb{T}$. We see that
\begin{align*}
 (-\Delta)^{-\frac12} \delta_0 &= \sum_{k \in \mathbb{Z} \atop k \neq 0}{(2\pi  \|k\|)^{-1} e^{2\pi i k x}} = \frac{1}{2\pi} \sum_{k \in \mathbb{Z} \atop k \neq 0}{\frac{e^{2\pi i k x}}{k}} \\
 &= \frac{1}{\pi} \sum_{k=1}^{\infty}{ \frac{\cos{(2 \pi k x)}}{k}} = - \frac{1}{\pi} \log{(2\sin{|\pi x|})}
\end{align*}
This is, up to a factor of $\pi$, exactly the factor arising in our computation. It is well understood that $s=-1/2$ is a special scale and that the fractional Laplacian has different behavior for $s<-1/2$ and $s>-1/2$ but it does suggest many other factors that can be computed in a similar way. It also suggests that it might be potentially worthwhile to study functionals of the type
$$ E(X) = \left\| \sum_{k=1}^{n} (-\Delta)^{s} \delta_{x_k} \right\|_{L^2(\mathbb{T}^d)}^2$$
which can be simplified
\begin{align*}
E(X) &= \left\langle \sum_{k=1}^{n} (-\Delta)^{s} \delta_{x_k}, \sum_{k=1}^{n} (-\Delta)^{s} \delta_{x_k} \right\rangle = \sum_{k, \ell=1}^{n}{ \left\langle (-\Delta)^{s} \delta_{x_k} , (-\Delta)^{s} \delta_{x_\ell} \right\rangle} \\
&= n \left\langle (-\Delta)^{s} \delta_{0}, (-\Delta)^{s} \delta_{0} \right\rangle + \sum_{k, \ell=1 \atop k \neq \ell}^{n}{ \left\langle (-\Delta)^{s} \delta_{x_k} , (-\Delta)^{s} \delta_{x_\ell} \right\rangle }
\end{align*}
Using self-adjointness of the inverse fractional Laplacian, we can simplify the relevant term as
\begin{align*}
 \sum_{k, \ell=1 \atop k \neq \ell}^{n}{ \left\langle (-\Delta)^{s} \delta_{x_k} , (-\Delta)^{s} \delta_{x_\ell} \right\rangle } &=  \sum_{k, \ell=1 \atop k \neq \ell}^{n}{ \left\langle (-\Delta)^{2s} \delta_{x_k} , \delta_{x_\ell} \right\rangle } =  \sum_{k, \ell=1 \atop k \neq \ell}^{n}{  \left( (-\Delta)^{2s} \delta_{0}\right)(x_k - x_{\ell})  }
\end{align*}
This, in turn, can be rewritten as 
$$  \sum_{k, \ell=1 \atop k \neq \ell}^{n}{  \left( (-\Delta)^{2s} \delta_{0}\right)(x_k - x_{\ell})  } = (2\pi)^{2s} \sum_{k, \ell=1 \atop k \neq \ell}^{n}{  \sum_{m \in \mathbb{Z}^d \atop m \neq 0}{ m^{2s} e^{2\pi i \left\langle m, x_k -x_{\ell}\right\rangle}}}$$
which, obviously, admits a gradient formulation. One could also consider a possible trunction in frequency followed by a gradient formulation as well as various mollification mechanism. We want to strongly suggest the possibility that the optimal value of $s$ for these kinds of methods may depend on the dimension.

\section{Heuristic Derivation of the Energy Functional}
We first give a one-dimensional argument to avoid notational overload and then derive the analogous quantity for higher dimensions in \S 3.2.
\subsection{One dimension.} Our derivation is motivated by the Erd\H{o}s-Turan inequality bounding the discrepancy $D_N$ of a set $\left\{x_1, \dots, x_N\right\} \subset [0,1]$ by
$$ D_N \lesssim \frac{1}{N} + \sum_{k=1}^{N}{ \frac{1}{k} \left| \frac{1}{N} \sum_{n=1}^{N}{ e^{2 \pi i k x_n} } \right|}.$$
We can bound this from above, using $x \leq (1+x^2)/2$ valid for all real $x$, by
$$ \sum_{k=1}^{N}{ \frac{1}{k} \left| \frac{1}{N} \sum_{n=1}^{N}{ e^{2 \pi i k x_n} } \right|} \leq \frac{1}{2}\sum_{k=1}^{N}{\left( \frac{1}{k}  \frac{1}{N} + \frac{1}{k} \frac{1}{N} \left| \sum_{n=1}^{N}{ e^{2 \pi i k x_n} } \right|^2 \right)}.$$
Using merely this upper bound, we want to make sure that the second term is small. This second term simplifies to 
$$ \frac{1}{N}\sum_{k=1}^{N}{\left(\frac{1}{k}  \left| \sum_{n=1}^{N}{ e^{2 \pi i k x_n} } \right|^2 \right)} = \frac{1}{N}\sum_{k=1}^{N}{\frac{1}{k}   \sum_{n, m=1}^{N}{ e^{2 \pi i k (x_n-x_m)} } }$$
Ignoring the scaling factor $N^{-1}$, we decouple into diagonal and off-diagonal terms and obtain 
$$ \sum_{k=1}^{N}{\frac{1}{k}   \sum_{n, m=1}^{N}{ e^{2 \pi i k (x_n-x_m)} } } =  \sum_{k=1}^{N}{\frac{N}{k}} +   \sum_{k=1}^{N}{\frac{1}{k}\sum_{m,n = 1 \atop m \neq n}^{N}{ \cos{(2 \pi k (x_m - x_n))}}}.$$
The first term is a fixed constant and thus independent of the actual points, the second sum can be written as
$$ \sum_{k=1}^{N}{\frac{1}{k}\sum_{m,n = 1 \atop m \neq n}^{N}{ \cos{(2 \pi k (x_m - x_n))}}} = \sum_{m,n = 1 \atop m \neq n}^{N}{  \sum_{k=1}^{N}{ \frac{ \cos{(2 \pi k (x_m - x_n))}}{k} }}.$$
The inner sum can now be simplified \cite{trig} by letting the limit go to infinity since
$$   \sum_{k=1}^{\infty}{ \frac{ \cos{(2 \pi k x)}}{k} } = - \log{(2 \sin{( \pi |x|)})}.$$
This suggests that we should really try to minimize the functional
$$ E(X) = \sum_{m,n = 1 \atop m \neq n}^{N}{  - \log{(2 \sin{( \pi |x_m - x_n|)})}}.$$

\textbf{Remark.} There is one step in the derivation where we could have argued somewhat differently: we could have written, for any $0 < \gamma < 1$,
\begin{align*}
\sum_{k=1}^{N}{ \frac{1}{k} \left| \frac{1}{N} \sum_{n=1}^{N}{ e^{2 \pi i k x_n} } \right|} &= \sum_{k=1}^{N}{ \frac{1}{k^{1-\gamma}}  \frac{1}{k^{\gamma}}\left| \frac{1}{N} \sum_{n=1}^{N}{ e^{2 \pi i k x_n} } \right|} \\
&\leq \left(\sum_{k=1}^{N}{ \frac{1}{k^{2-2\gamma}}} \right)^{1/2} \frac{1}{N} \left( \sum_{k=1}^{N}{ \frac{1}{k^{2\gamma}} \left| \sum_{n=1}^{N}{ e^{2 \pi i k x_n} } \right|^2} \right)^{1/2}.
\end{align*}
The first sum simplifies to either to $\sim N^{\gamma - 1/2}$ (for $\gamma > 1/2$), to $\sim \log{N}$ (for $\gamma =1/2$) or to $\sim 1$ (for $\gamma < 1/2$). The second term simplifies, after squaring the inner term and taking the limit of $N \rightarrow \infty$ over the Fourier series, to the definition of the fractional Laplacian $(-\Delta)^{-\gamma}$ (see \S 2.4.) applied to the measure $\sum_{k=1}^{N}{\delta_{x_k}}$.

\subsection{Higher dimensions.} The general case follows from the Erd\H{o}s-Turan-Koksma \cite{erd1,erd2,koksma} inequality and the heuristic outlined above for the one-dimensional case. We recall that the Erd\H{o}s-Turan-Koksma inequality allows us to bound the discrepancy of a set $\left\{x_1, \dots, x_N\right\} \subset [0,1]^d$ by

$$ D_N \lesssim_d \frac{1}{M + 1} + \sum_{\|k\|_{\infty} \leq M}{ \frac{1}{r(k)} \frac{1}{N} \left| \sum_{\ell=1}^{N}{e^{2\pi i \left\langle k, x_{\ell}  \right\rangle}} \right|},$$
where $r:\mathbb{Z}^d \rightarrow \mathbb{N}$ is given by
$$ r(k) = \prod_{j=1}^{d}{\max\left\{1, |k_j| \right\}}.$$
We note that, since $r(k) \leq r(2k) \leq 2^d r(k)$, we can change $r(k)$ to $r(2k)$ at merely the cost of a constant depending only on the dimension and thus
\begin{align*}
 \sum_{\|k\|_{\infty} \leq M}{ \frac{1}{r(k)} \frac{1}{N} \left| \sum_{\ell=1}^{N}{e^{2\pi i \left\langle k, x_{\ell}  \right\rangle}} \right|} &\lesssim_d  \sum_{\|k\|_{\infty} \leq M}{ \frac{1}{r(2k)} \frac{1}{N}} \\
&+\sum_{\|k\|_{\infty} \leq M}{ \frac{1}{r(2k)} \frac{1}{N} \left| \sum_{\ell=1}^{N}{e^{2\pi i \left\langle k, x_{\ell}  \right\rangle}} \right|^2}.
\end{align*}
The second sum we can expand into
$$ \sum_{\|k\|_{\infty} \leq M}{ \frac{1}{r(2k)} \frac{1}{N} \left| \sum_{\ell=1}^{N}{e^{2\pi i \left\langle k, x_{\ell}  \right\rangle}} \right|^2} = \frac{1}{N}\sum_{m,n = 1}^{N} \prod_{j=1}^{d} \left( 1 + \sum_{k=-M \atop k \neq 0}^{M}{\frac{1}{2|k|}  e^{2\pi i k (x_{m,j} - x_{n,j})}} \right).$$
Letting $M \rightarrow \infty$, we can simplify every one of these terms to
$$\sum_{k \in \mathbb{Z} \atop k \neq 0}^{\infty}{\frac{e^{2\pi i k (x_{m,j} - x_{n,j})}}{2|k|}  } =  -  \log{(2 \sin{(\pi |x_{m,j} - x_{n,j}|)})}$$
and we obtain the general form of the energy functional.

\begin{figure}[h!]
\begin{minipage}[l]{.49\textwidth}
\includegraphics[width = 5cm]{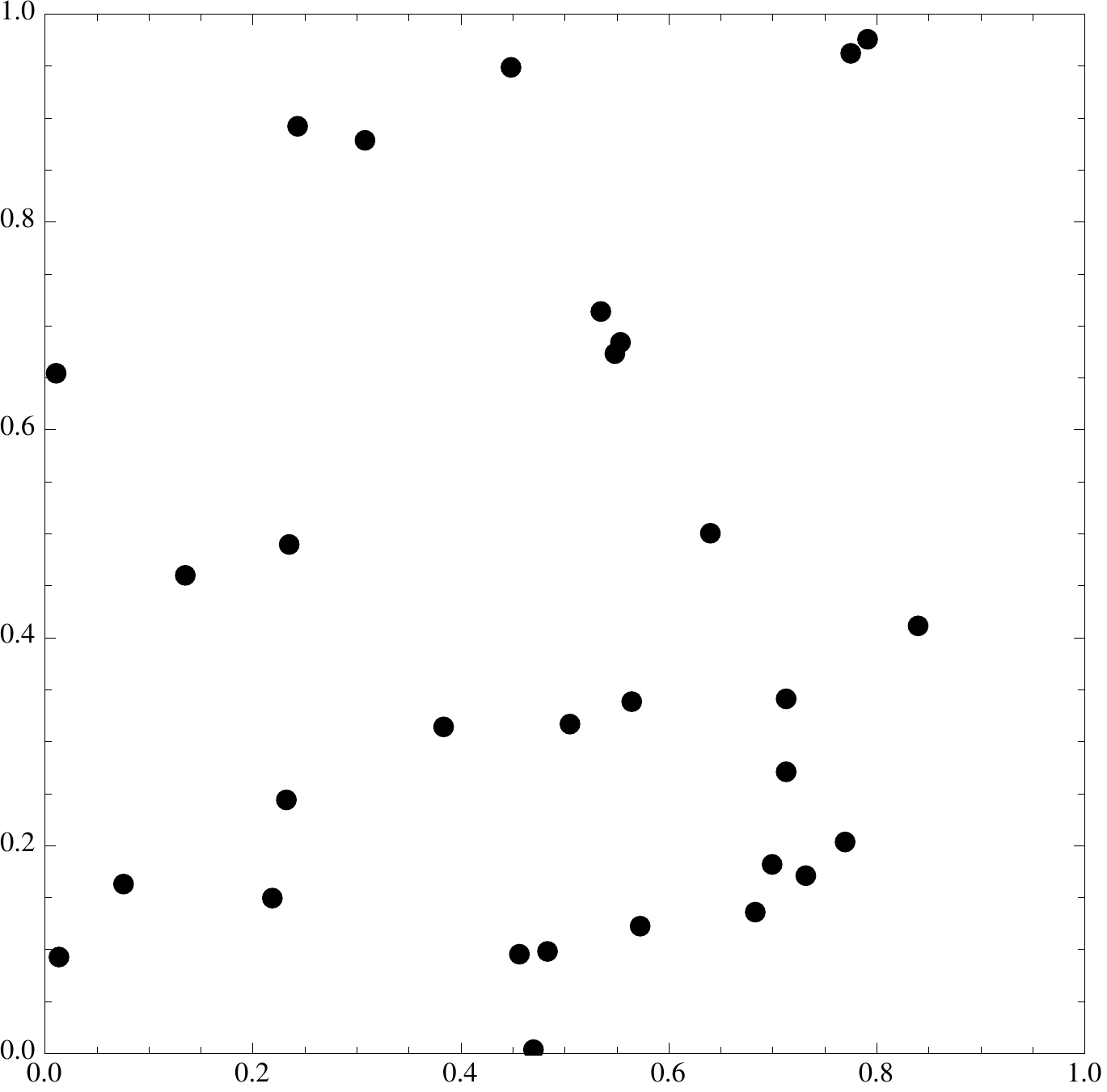} 
\end{minipage} 
\begin{minipage}[r]{.49\textwidth}
\includegraphics[width = 5cm]{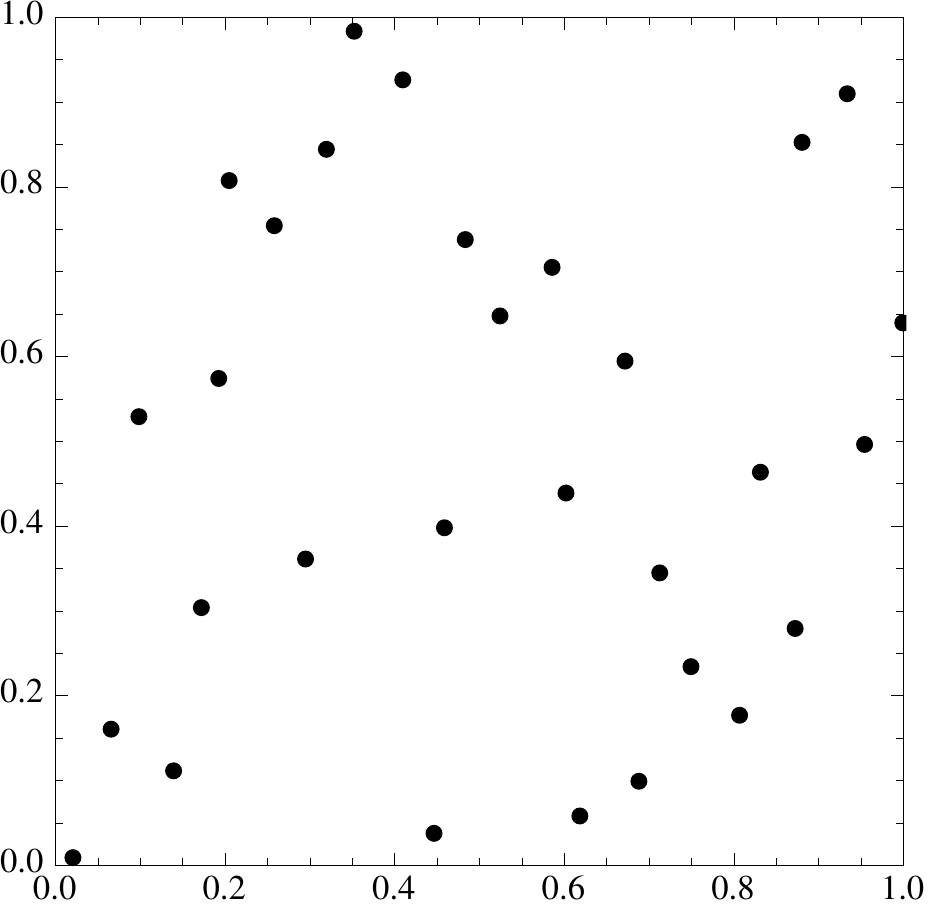} 
\end{minipage} 
\caption{Left: 100 random points with $D_N^* \sim 0.12$. Right: evolution of the gradient flow leads to a set with $D_N^* \sim 0.05$.}
\end{figure}

The Erd\H{o}s-Turan-Koksma inequality shows
$$ D_N \lesssim_d \frac{1}{M + 1} + \sum_{\|k\|_{\infty} \leq M}{ \frac{1}{r(k)} \frac{1}{N} \left| \sum_{\ell=1}^{N}{e^{2\pi i \left\langle k, x_{\ell}  \right\rangle}} \right|}.$$
We know that the best possible behavior is on the scale of $D_N \lesssim (\log{N})^{d-1} N^{-1}$ (or possibly even smaller). This suggests that the exponential sums cannot typically be that large, it should be roughly at scale $\sim 1$ most of the time. Understanding this better could lead
to precise estimates comparing how much our energy exceeds the discrepancy.
We conclude by establishing a rigorous bound.
\begin{lem} We have, for $X=\left\{x_1, \dots, x_N \right\} \subset \mathbb{T}^d$,
$$\sum_{\|k\|_{\infty} \leq N}{ \frac{1}{r(k)}\left| \sum_{\ell=1}^{N}{e^{2\pi i \left\langle k, x_{\ell}  \right\rangle}} \right|^2} \lesssim_d E(X)$$
\end{lem}
\begin{proof} The argument outlined above already establishes the result except for one missing ingredient: for all $0 < x < 1$, there is a uniform bound
$$ \max_{n \in \mathbb{N}} \sum_{k=1}^{n}{\frac{\cos{(2\pi k x)}}{k}} \lesssim 1-\log{|\sin{(\pi x)}|}.$$
We can assume w.l.o.g. that $0 < x < 1/2$. We use Abel summation to write
\begin{align*}
\sum_{k=1}^{n}{\frac{\cos{(2\pi k x)}}{k}} &= (n+1)\frac{\cos{(2\pi n x)}}{n} \\
&+ \int_{1}^{n}{ \left\lfloor k+1 \right\rfloor \left( \frac{\cos{(2 \pi k x)}}{k^2} + \frac{ 2\pi x  \sin{(2\pi k x)}}{k} \right) dk}.
\end{align*}
The first term is $\mathcal{O}(1)$, it remains to treat the integral. The first term has the structure of an alternating Leibniz series with
the first root being at $kx = 1/4$. Thus
\begin{align*}
 \int_{1}^{n}{ \left\lfloor k+1 \right\rfloor  \frac{\cos{(2 \pi k x)}}{k^2} dk} &\lesssim \int_{1}^{1/(4x)}{ \left\lfloor k+1 \right\rfloor  \frac{\cos{(2 \pi k x)}}{k^2} dk} \\
&\lesssim \int_{1}^{1/(4x)}{  \frac{\cos{(2 \pi k x)}}{k} dk} \lesssim \log{(1/x)}.
\end{align*}
The second integral simplifies to
$$ \int_{1}^{n}{ \left\lfloor k+1 \right\rfloor  \frac{ 2\pi x  \sin{(2\pi k x)}}{k}  dk} = 2\pi x \int_{1}^{n}{ \left\lfloor k+1 \right\rfloor  \frac{\sin{(2\pi k x)}}{k}  dk} \lesssim 1.$$
\end{proof}

\begin{figure}[h!]
\begin{minipage}[l]{.49\textwidth}
\includegraphics[width = 5.5cm]{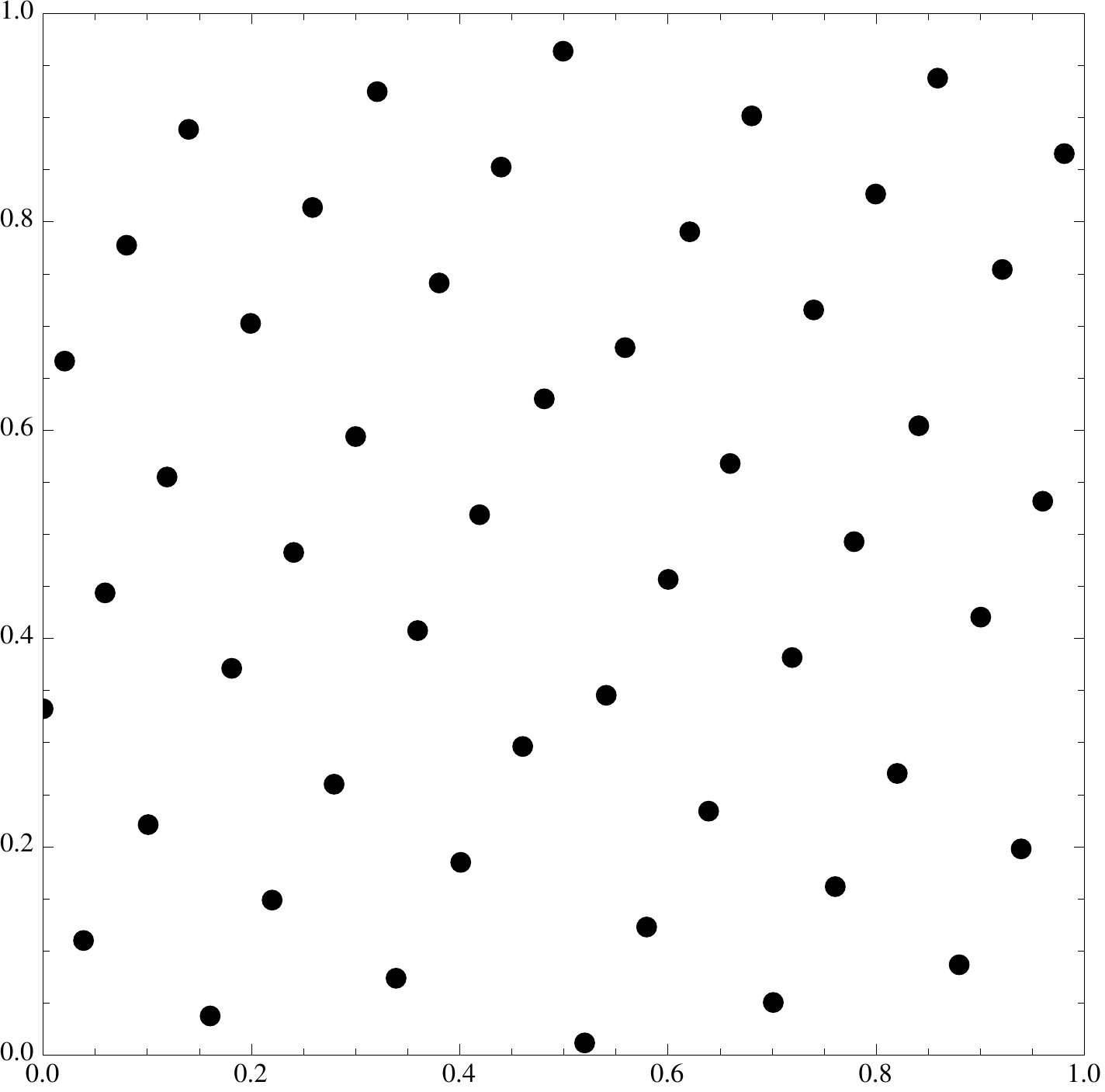} 
\end{minipage} 
\begin{minipage}[r]{.49\textwidth}
\includegraphics[width = 5.5cm]{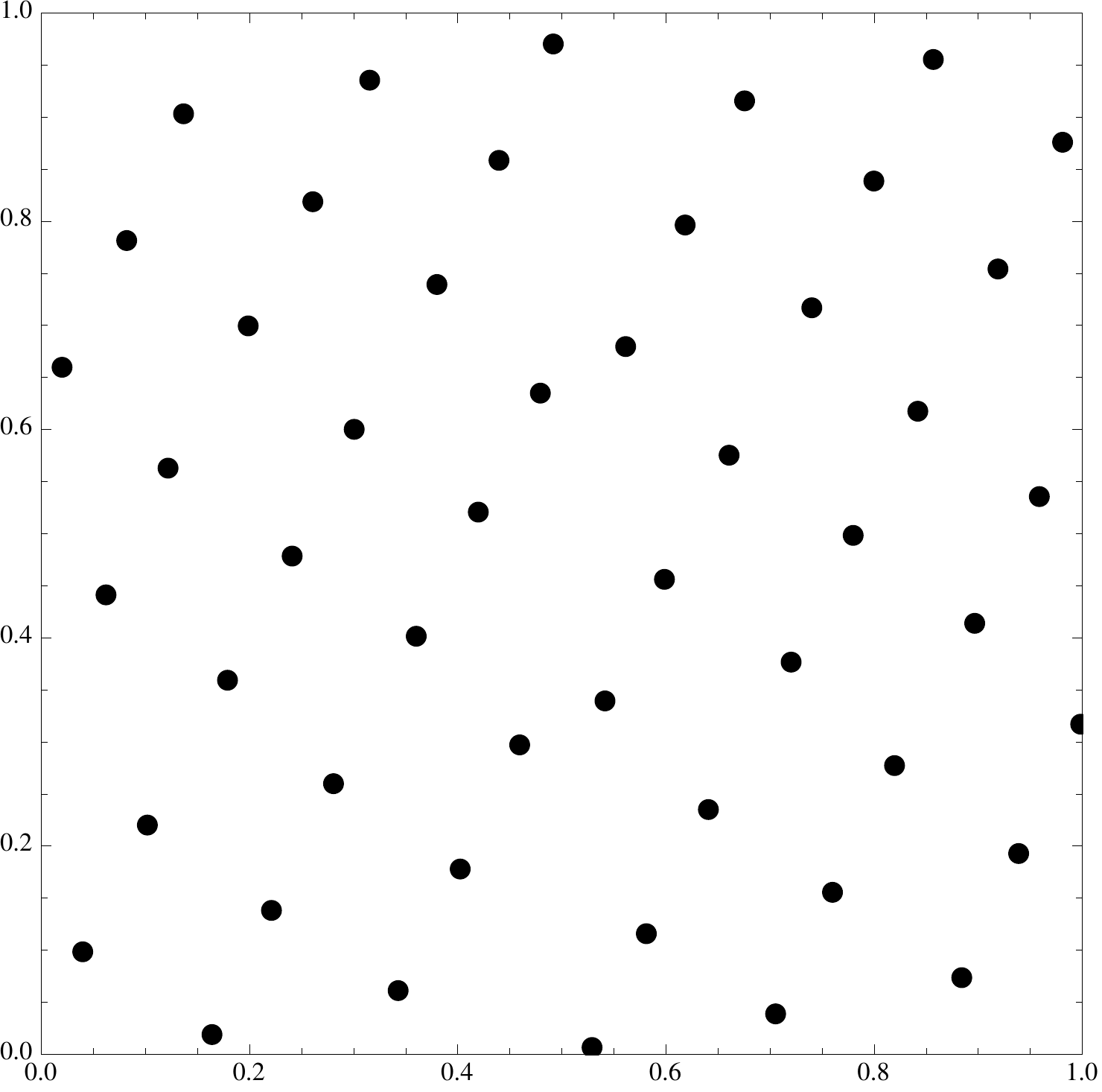} 
\end{minipage} 
\caption{Left: 50 points of the Hammersley sequence in base 3 with $D_N^* \sim 0.064$. Right: evolution of the flow leads to a set with $D_N^* \sim 0.042$.}
\end{figure}

\subsection{The case $d=1$.} Things are usually simpler in one dimension (though also less interesting because the optimal constructions are trivial and
given by equispaced points). We have the following basic result.

\begin{proposition} Let $(x_n)$ be a sequence in $\mathbb{T} \cong [0,1]$. If
$$ \limsup_{N \rightarrow \infty} \frac{1}{N^2} \sum_{1 \leq m \neq n \leq N}{ (1 - \log{(2 \sin{(\pi |x_m - x_n|)})})} = 1,$$
then the sequence is uniformly distributed.
\end{proposition}
\begin{proof} The proof is similar in spirit to the main argument in \cite{steini}, we refer to that paper for definition of the Jacobi $\theta-$function and the main idea.
We define a one-parameter family of functions via
$$ f_N(t,x) = \sum_{k=1}^{N}{ \theta_t(x-x_k)}$$
where $\theta_t$ is the Jacobi $\theta-$function. In particular
$$ \lim_{t \rightarrow 0^+}{ f_N(t,x)} = \frac{1}{N}\sum_{k=1}^{N}{\delta_{x_k}} \qquad \mbox{in the sense of weak convergence.}$$
Defining 
$$ g(x) = 1 - \log{(2 \sin{( \pi |x|)} )},$$
we can define the function
$$h(t) = \left\langle g* f_N(t,x), f_N(t,x) \right\rangle$$
is monotonically decaying in time. This is seen by applying the Plancherel theorem
\begin{align*}
 h(t) &= \sum_{k \in \mathbb{Z}} \widehat{g}(k) |\widehat{ f_N(t,x)}(k)|^2 \\
&=  \sum_{k \in \mathbb{Z}}{ \widehat{g}(k)  e^{-4 \pi^2 k^2 t} \left|\sum_{\ell=1}^{N}{e^{-2 \pi i  k x_{\ell}}} \right|^2}
\end{align*}
and using $ \widehat{g}(k) > 0$. We can now take the limit $t \rightarrow \infty$ and obtain that
$$ h(t) \geq \widehat{g}(0) N^2 = N^2.$$
As for the second part of the argument, suppose that $(x_n)$ is not uniformly distributed. Weyl's theorem implies that there exist $\varepsilon>0, k \in \mathbb{N}$ such that
$$ \left| \sum_{\ell=1}^{N}{ e^{-2 \pi i k x_{\ell}}} \right|^2 \geq \varepsilon \qquad \mbox{for infinitely many}~n.$$
Then, however,
$$ h(1) \geq \widehat{g}(0) N^2 + e^{-4 \pi^2 k^2} |\widehat{g}(k)|^2 \geq (1 + \delta)N^2$$
for some $\delta > 0$ and infinitely many $N$.
\end{proof}

\section{Proof of the Theorem}
\subsection{An Inequality.} We first prove an elementary inequality.
\begin{lem} Let $0 < x,y < 1$. Then
$$ 2 \left|\cot{(\pi x)} \cot{(\pi y)} \right|< (1-\log{(2\sin{(\pi x)})}) \csc^2{(\pi y)} +   \csc^2{(\pi x)}(1-\log{(2\sin{(\pi y)})}).$$
\end{lem}
\begin{proof} The right-hand side is always positive, we can thus assume w.l.o.g. that $0 < x,y < 1/2$. Multiplying with $\sin^2{(\pi x)}\sin^2{(\pi y)}$ on both sides leads to the equivalent 
statement $A \leq B$, where
$$ A = 2\sin{(\pi x)}\cos{(\pi x)}  \sin{(\pi y)}\cos{(\pi y)}$$
and 
$$B = (1-\log{(2\sin{(\pi x)})}) \sin^2{(\pi x)} +   \sin^2{(\pi y)}(1-\log{(2\sin{(\pi y)})}).$$
We use $2ab \leq a^2 + b^2$ to argue that
\begin{align*}
A \leq \sin^2{(\pi x)}\cos^2{(\pi x)} +  \sin^2{(\pi y)}\cos^2{(\pi y)}.
\end{align*}
The result then follows from the inequality
$$ \cos^2{(\pi x)} < 1-\log{(2\sin{(\pi x)})} \qquad \mbox{for all}~0 < x \leq \frac{1}{2}$$
which can be easily seen by elementary methods.
\end{proof}


\subsection{Proof of the Theorem}
\begin{proof}[Proof of the Theorem.] The symmetries of the sequence and the energy functional imply that it is sufficient to show that the energy is locally convex around the point in $(0,0)$. This means we want to show that
$$f(\varepsilon, \delta) = \sum_{n=1}^{N-1}{     \left(1 - \log{\left(2 \sin{ \left( \pi \left|\frac{n}{N} - \varepsilon \right|\right)} \right)}\right)  \left(1 - \log{\left(2 \sin{ \left( \pi \left| \left\{ \frac{an}{N}\right\} - \delta \right|\right)} \right)} \right) }$$
is strictly positive for all $\varepsilon, \delta$ sufficiently small. We can assume $|\varepsilon|, |\delta| < N^{-1}$, expand the first term in $\varepsilon$ up to second order and note that
\begin{align*}
 1 - \log{( 2 \sin{(\pi (x - \varepsilon))})}&= \left(1 - \log{( 2 \sin{(\pi x)})}\right) + \pi \cot{(x \pi)} \varepsilon \\
&+   \pi^2 \csc^2{(\pi x)} \frac{ \varepsilon^2}{2} + \mathcal{O}(\varepsilon^3).
\end{align*}
This shows that
$$ \frac{\partial}{\partial \varepsilon} g(\varepsilon, 0) \big|_{\varepsilon = 0} =  \sum_{n=1}^{N-1}{   \cot{ \left(\frac{n  \pi }{N}  \right)}  \left(1 - \log{\left(2 \sin{ \left( \pi  \left\{ \frac{an}{N}\right\} \right)} \right)} \right) }.$$
We group the summand $n$ and $N-n$ and observe that $\cot$ is odd on $(0, \pi)$ while the second summand is even, therefore the sum evaluates to 0. The other derivative
$$ \frac{\partial}{\partial \delta} g(0, \delta) \big|_{\delta = 0} =  \sum_{n=1}^{N-1}{   \cot{ \left(  \pi \left\{ \frac{an}{N}\right\}  \right)}  \left(1 - \log{\left(2 \sin{ \left( \pi \frac{n  }{N} \right)} \right)} \right) }$$
vanishes for exactly the same reason and therefore the lattice is a critical point. It remains to show that it is a local minimizer which requires an expansion up to second order. 
This expansion naturally decouples into three sums, where
\begin{align*}
(I) &=  \frac{\pi^2 \varepsilon^2}{2} \sum_{n=1}^{N-1}{    \csc^2{ \left( \pi\frac{n}{N}  \right)}  \left(1 - \log{\left(2 \sin{ \left( \pi \left\{ \frac{an}{N}\right\}  \right)} \right)} \right) }\\
(II) &= \pi^2 \varepsilon \delta \sum_{n=1}^{N-1}{ \cot{  \left( \pi\frac{n}{N}  \right)}\cot{  \left( \left\{ \pi\frac{an}{N} \right\}  \right)}} \\
(III) &=\frac{\pi^2 \delta^2}{2} \sum_{n=1}^{N-1}{  \csc^2{ \left( \pi \left\{ \frac{an}{N} \right\} \right)}    \left(1 - \log{\left(2 \sin{ \left( \pi  \frac{n}{N} \right)}\right)}\right)  }  
\end{align*}

\begin{figure}[h!]
\begin{minipage}[l]{.49\textwidth}
\includegraphics[width = 5cm]{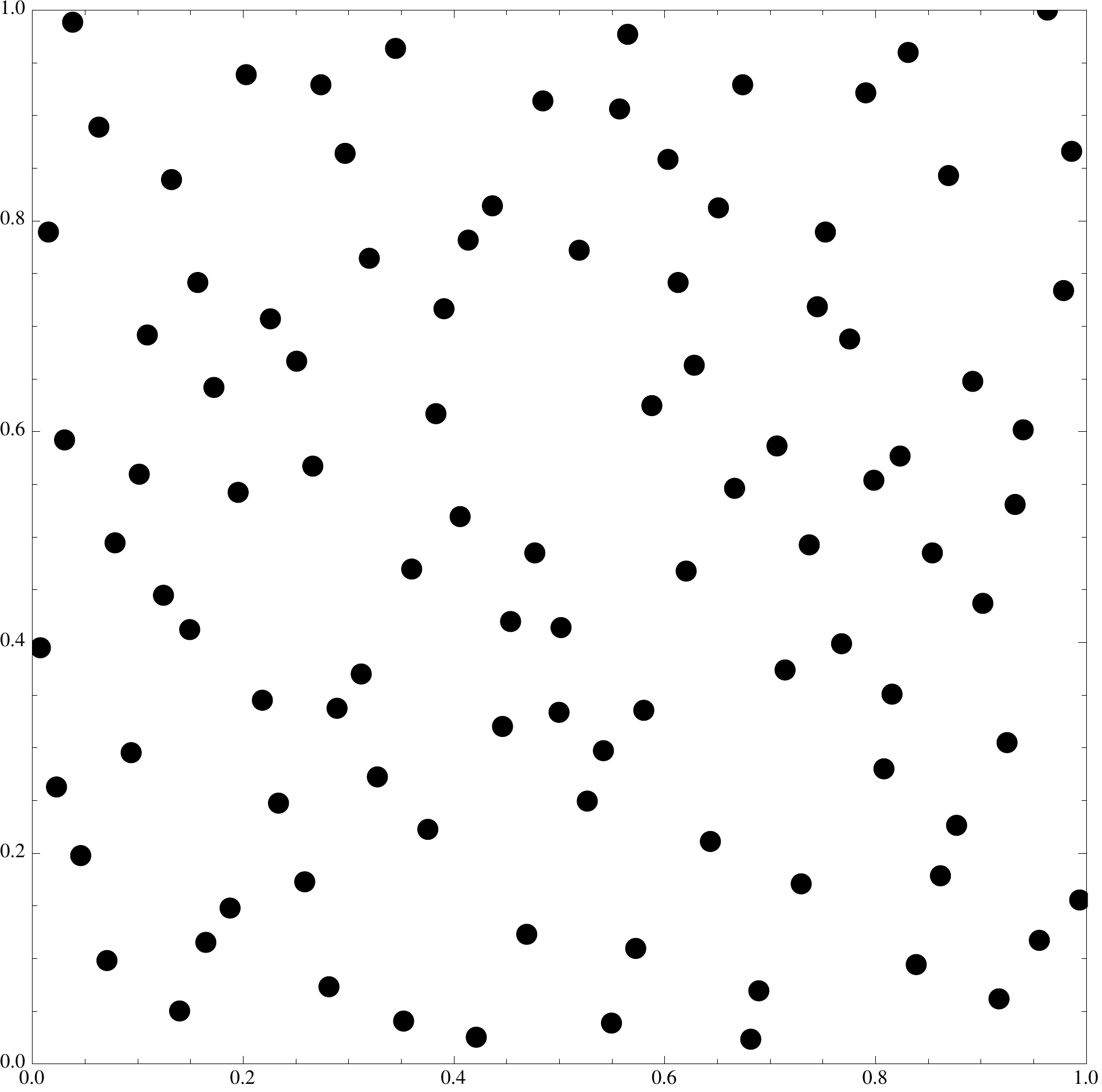} 
\end{minipage} 
\begin{minipage}[r]{.49\textwidth}
\includegraphics[width = 5cm]{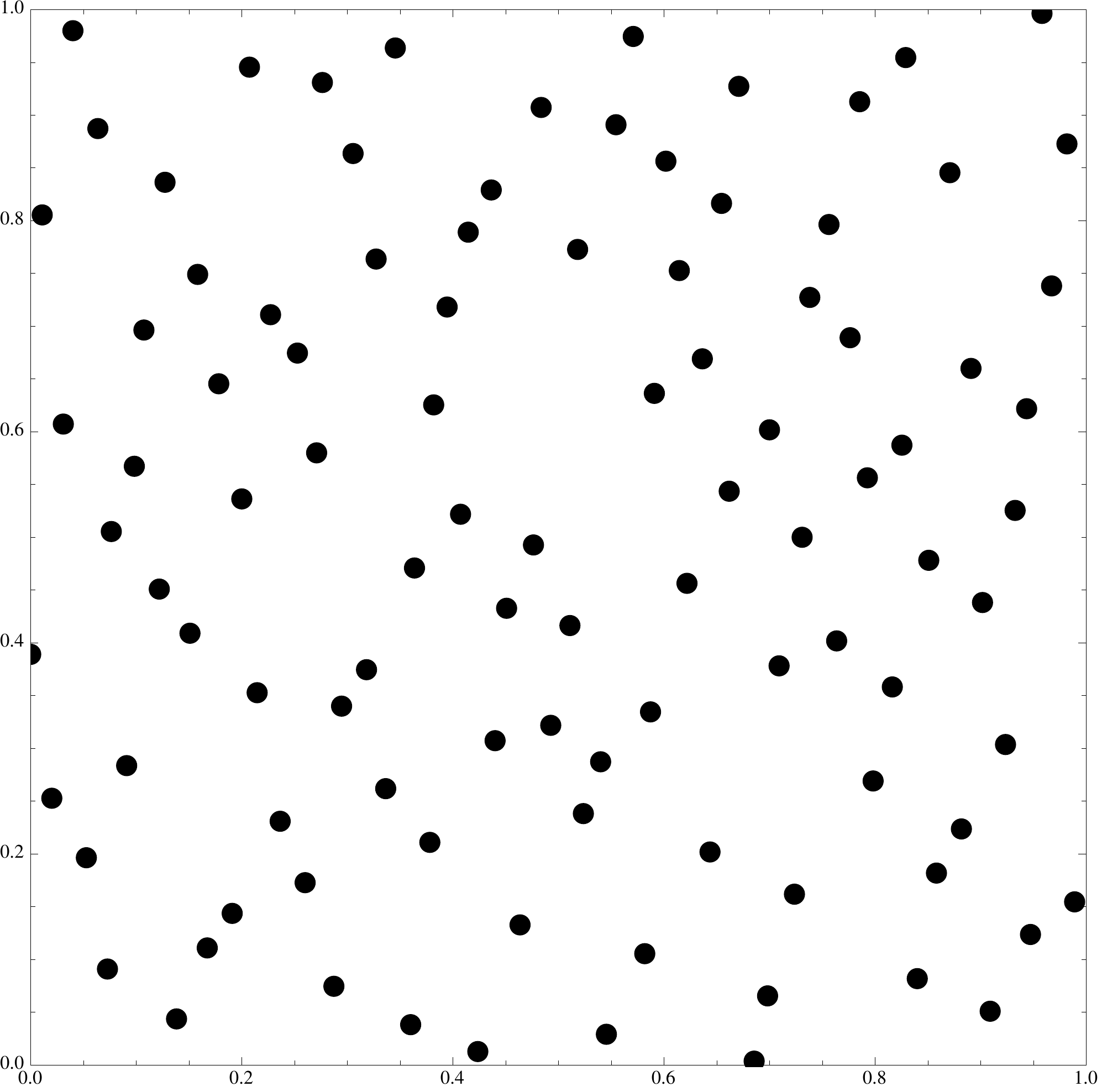} 
\end{minipage} 
\caption{Left: 101 points combined from a Halton sequence $(x \leq 0.5)$ and a Sobol sequence $(x \geq 0.5)$ with $D_N^* \sim 0.042$. Right: gradient flow leads to a set with  discrepancy $D_N^* \sim 0.034$. .}
\end{figure}

We can now argue that $(II)$ is bounded by
\begin{align*}
\left| \sum_{n=1}^{N-1}{ \varepsilon \delta \cot{  \left( \pi\frac{n}{N}  \right)}\cot{  \left( \left\{ \pi\frac{an}{N} \right\}  \right)}} \right| &\leq \left(\frac{\varepsilon^2}{2} + \frac{\delta^2}{2}\right)  \left|\sum_{n=1}^{N-1}{ \cot{  \left( \pi\frac{n}{N}  \right)}\cot{  \left( \left\{ \pi\frac{an}{N} \right\}  \right)}} \right|\\
&\leq \frac{ \varepsilon^2}{2}  \left|  \sum_{n=1}^{N-1}  \cot{  \left( \pi\frac{n}{N}  \right)}\cot{  \left( \left\{ \pi\frac{an}{N} \right\} \right) } \right| \\
&+ \frac{ \delta^2}{2} \left| \sum_{n=1}^{N-1}  \cot{  \left( \pi\frac{n}{N}  \right)}\cot{  \left( \left\{ \pi\frac{an}{N} \right\} \right) } \right|.
\end{align*}
The Lemma implies that we can bound the first term by
\begin{align*}
 \sum_{n=1}^{N-1}  \left|  \cot{  \left( \pi\frac{n}{N}  \right)}\cot{  \left( \left\{ \pi\frac{an}{N} \right\} \right) } \right|  &\leq \frac{1}{2}\sum_{n=1}^{N-1}   \csc^2{ \left( \pi\frac{n}{N}  \right)}  \left(1 - \log{\left(2 \sin{ \left( \pi \left\{ \frac{an}{N}\right\}  \right)} \right)} \right) \\
&+ \frac{1}{2}\sum_{n=1}^{N-1}   \csc^2{ \left( \pi  \left\{ \frac{an}{N}\right\} \right)}  \left(1 - \log{\left(2 \sin{ \left( \pi\frac{n}{N}  \right)} \right)} \right)
\end{align*}

We finally use the algebraic structure and argue that if $a^2 \equiv 1~(\mbox{mod}~N)$, then
$$ n \rightarrow a \cdot n \qquad \mbox{is an involution mod}~N$$
and that implies that both sums are actually the same sum written in a different order.
The arising sum is actually the term we are given in $(I)$. The argument for the third sum is identical and altogether we conclude that
$$ (II) \leq (I) + (III)$$
which implies the desired result.
\end{proof}

It remains an open question whether the same result $ (II) \leq (I) + (III)$ remains true in general. Basic numerical experiments suggest that this should be the case.
We can reformulate the problem by writing out the quadratic form and computing its determinant. The relevant question is then whether $ (1)(2) \geq (3)^2,$
where
\begin{align*}
(1) &=  \sum_{n=1}^{N-1}{    \csc^2{ \left( \pi\frac{n}{N}  \right)}  \left(1 - \log{\left(2 \sin{ \left( \pi \left\{ \frac{an}{N}\right\}  \right)} \right)} \right) }\\
(2) &=    \sum_{n=1}^{N-1}{ \cot{  \left( \pi\frac{n}{N}  \right)}\cot{  \left( \left\{ \pi\frac{an}{N} \right\}  \right)}} \\
(3) &= \sum_{n=1}^{N-1}{  \csc^2{ \left( \pi \left\{ \frac{an}{N} \right\} \right)}    \left(1 - \log{\left(2 \sin{ \left( \pi  \frac{n}{N} \right)}\right)}\right)  }.
\end{align*}


\begin{thebibliography}{10}

\bibitem{aisti} C. Aistleitner, Covering numbers, dyadic chaining and discrepancy, Journal of Complexity 27(6), p. 531-540, 2011.

\bibitem{ba} A. Baernstein II, A minimum problem for heat kernels of flat tori. Extremal Riemann surfaces (San Francisco, CA, 1995), 227--243, 
Contemp. Math., 201, Amer. Math. Soc., Providence, RI, 1997. 

\bibitem{bil1}  D. Bilyk, Roth's orthogonal function method in discrepancy theory, Uniform Distribution Theory 6 (2011), no. 1, 143--184.


\bibitem{bil3} D. Bilyk and M. Lacey, On the small ball Inequality in three dimensions, Duke Math. J. 143 (2008), no. 1, 81--115.

\bibitem{bil4}  D. Bilyk, M. Lacey and A. Vagharshakyan, On the small ball inequality in all dimensions, J. Funct. Anal. 254 (2008), no. 9, 2470--2502.


\bibitem{dick} J. Dick and F. Pillichshammer, Digital Nets and
  Sequences. Discrepancy theory and quasi-Monte Carlo
  integration. Cambridge University Press, Cambridge, 2010.
  
\bibitem{fritz} J. Dick and F. Pillichshammer, Discrepancy Theory and Quasi-Monte Carlo Integration,  in: A Panorama of Discrepancy Theory, Springer, 2014.
  
  \bibitem{fritz2} J. Dick, D. Nuyens, F. Pillichshammer, Lattice rules for nonperiodic smooth integrands, Numerische Mathematik 126, p. 259-291, 2014
  
\bibitem{doerr0} B. Doerr, M. Gnewuch, A. Srivastav, Bounds and constructions for the star-discrepancy via $\delta$-covers, Journal of Complexity 21(5), p. 691-709 (2005)

\bibitem{doerr00} B. Doerr, A lower bound for the discrepancy of a random point set, Journal of Complexity 30, p. 16-20, 2014.

\bibitem{doerr2} C. Doerr, M. Gnewuch, M. Wahlstr\"om, Calculation of Discrepancy Measures and Applications, in: A Panorama of Discrepancy Theory,  Springer, 2014.

\bibitem{drmota} M. Drmota, R. Tichy,  Sequences, Discrepancies and Applications. Lecture Notes in Mathematics, 1651. Springer-Verlag, Berlin, 1997.

\bibitem{erd1} P. Erd\H{o}s, P. Turan, On a problem in the theory of uniform distribution. I. Nederl. Akad. Wetensch. 51: p. 1146--1154, 1948.

\bibitem{erd2} P. Erd\H{o}s, P. Turan, On a problem in the theory of uniform distribution. II. Nederl. Akad. Wetensch. 51: p. 1262--1269, 1948.


\bibitem{gnewuch} M. Gnewuch, M. Wahlstr\"om, C. Winzen, A new randomized algorithm to approximate the star discrepancy based on threshold accepting, SIAM Journal on Numerical Analysis 50, p. 781-807, 2012.

\bibitem{gnewuch2}  M. Gnewuch, A. Srivastav, C. Winzen, Finding optimal volume subintervals with k points and calculating the star discrepancy are NP-hard problems, Journal of Complexity 25(2), p. 115-127, 2009.

\bibitem{halton} J. Halton,  On the efficiency of certain quasi-random sequences of points in evaluating
multi-dimensional integrals, Numer. Math. 2 (1960), 84--90.

\bibitem{hammersley} J. M. Hammersley and D. C. Handscomb, Monte Carlo Methods, Methuen, London, 1964.

\bibitem{heinrich} S. Heinrich, E. Novak, G. W. Wasilkowski and H. Wozniakowski, The inverse of the star
discrepancy depends linearly on the dimension, Acta Arith. 96 (2001), pp. 279--302.

\bibitem{hinrich} A. Hinrichs, Covering numbers, Vapnik-Cervonenkis classes and bounds for the star-discrepancy, Journal of Complexity 20, p. 477-483 (2004).

\bibitem{koksma} J. F. Koksma. Some theorems on diophantine inequalities.Math. Cen-trum Amsterdam Scriptum 5, 1950

\bibitem{fritz3} P. Kritzer and F. Pillichshammer, Low discrepancy polynomial lattice point sets, Journal of Number Theory 132, p. 2510-2534, 2012.

\bibitem{kuipers} L. Kuipers and H. Niederreiter, Uniform Distribution of Sequences. Pure and Applied Mathematics. Wiley-Interscience, New York-London-Sydney, 1974.

\bibitem{mat} J.  Matousek: Geometric Discrepancy. Springer, Berlin Heidelberg New York, 1999.

\bibitem{trig} C.-S. Lin, On log-trigonometric functions, Crux Mathematicorum 29(7), p, 460--463, 2003.

\bibitem{mo} H. Montgomery, Minimal theta functions. Glasgow Math. J. 30 (1988), no. 1, 75--85. 

\bibitem{nied} H. Niederreiter, Quasi-Monte Carlo Methods and Pseudo-Random Numbers, Bull. Amer. Math. Soc. 84, p. 957--1041, 1978 

\bibitem{niederreiter} H. Niederreiter, Point sets and sequences with small discrepancy, Monatshefte f\"ur Mathematik 104, p. 273--337 (1987) 

\bibitem{sachs} J. Lu, M. Sachs and S. Steinerberger, Quadrature Points via Heat Kernel Repulsion,  Constructive Approximation, to appear

\bibitem{os} B.
Osting and J. Marzuola, 
Spectrally optimized pointset configurations.  
Constr. Approx. 46 (2017), no. 1, 1--35. 

\bibitem{luz} L. Roncal, P. R. Stinga, Fractional Laplacian on the torus, Commun. Contemp. Math. 18, 1550033, 26 pp, (2016).

\bibitem{sch} W. Schmidt, Irregularities of distribution. VII.  Acta Arith. 21 (1972), 45--50. 

\bibitem{sobol} I. M. Sobol, Distribution of points in a cube and approximate evaluation of integrals. Zh. Vych. Mat. Mat. Fiz. 7: 784--802 (in Russian); U.S.S.R Comput. Maths. Math. Phys. 7: 86--112 (in English), 1967.

\bibitem{steini} S. Steinerberger, Localized Quantitative Criteria for Equidistribution, Acta Arithmetica, 180, 183--199, 2017. 

\bibitem{stein} S. Steinerberger, Spectral Limitations of Quadrature Rules and Generalized Spherical Designs, arXiv:1708.08736

\bibitem{steind} S. Steinerberger, Dynamically Defined Sequences with Small Discrepancy,  arXiv:1902.03269

\bibitem{young} W. H. Young, On a certain series of Fourier, Proc. London Math. Soc. 11 (1912), 357--366


\end{thebibliography}
\end{document}